\documentclass[reqno,11pt,draft]{amsart}
\usepackage{amssymb}
\usepackage{amsmath}
\usepackage{amsthm}
\usepackage{url}
\usepackage{graphicx}

\usepackage{mathptmx}       
\usepackage{helvet}         
\usepackage{courier}        

 \newtheorem{lemma}{Lemma}
\newtheorem{proposition}{Proposition}
 \newtheorem{theorem}{Theorem}
 \newtheorem{corollary}{Corollary}
 
 \theoremstyle{definition}
 \newtheorem{definition}{Definition}
 
 \theoremstyle{remark}
 \newtheorem{remark}{Remark}

\newcommand{\mc}[1]{{\mathcal #1}}

\newcommand{\bb}[1]{{\mathbb #1}}

\newcommand{\<}{\langle}
\renewcommand{\>}{\rangle}

\begin{document}

\title[Anomalous fluctuations]{Current and density fluctuations for interacting particle systems with anomalous diffusive behavior}

\author{M. Jara}
\address{CEREMADE \\ Universit\'e Paris-Dauphine\\
Place du Mar\'echal de Lattre de Tassigny\\
Paris CEDEX 75775\\
France\\
}
\email{jara@ceremade.dauphine.fr}


\begin{abstract}
We prove density and current fluctuations for two examples of symmetric, interacting particle systems with anomalous diffusive behavior: the zero-range process with long jumps and the zero-range process with degenerated bond disorder. As an application, we obtain subdiffusive behavior of a tagged particle in a simple exclusion process with variable diffusion coefficient.
\end{abstract}

\maketitle

\section{Introduction}

Since the work of Harris \cite{Har}, it is known that the motion of a tagged or distinguished particle in symmetric, diffusive, one-dymensional systems of particles that preserve the relative order of particles is subdiffusive, in the sense that the mean square displacement $E[x_0(t)^2]$ grows with $t$ as $t^{1/2}$, where $x_0(t)$ denotes the position of a tagged particle, initially at the origin.  This is a much slower rate than the linear growth obtained for usual diffusions. In Harris' original work, a system of independent Brownian motions with reflection was considered. These kind of systems are known in the physics literature as {\em single-file diffusions} (see \cite{BS} for a recent discussion and further references). 

In this subdiffusive setting, it has been proved that the rescaled position of the tagged particle converges, in the sense of finite-dimensional distributions, to a fractional Brownian motion (\cite{Arr}, \cite{DF}, \cite{JL} for the simple exclusion process; \cite{RV} for interacting Brownian motions). Recently, a functional central limit theorem has been obtained \cite{PS}. This subdiffusive behavior is characteristic of single-file diffusions; when particles can pass one over the others or in dimension $d>1$, the tagged particle converges to a Brownian motion \cite{KV}. For biased particles, the limit is also diffusive \cite{K}.

From the work of Rost and Vares, we know that the displacement $x_0(t)$ of the tagged particle can be identified as the mass-current through the origin in the {\em increment process} associated to the single-file diffusion. Let $\{x_i(t);i \in \bb Z\}$ be the position of the particles at time $t \geq 0$ in a single-file diffusion. We assume that $x_i(t) \leq x_{i+1}(t)$. Define $\eta_t(i) = x_{i+1}(t)-x_i(t)-\epsilon$, where $\epsilon=0$ in the case of particles evolving on the real line, and $\epsilon =1$ in the case of particles evolving on the lattice. The process $\eta_t =\{\eta_t(i);i \in \bb Z\}$ turns out to have a Markovian evolution with a local dynamics. Let $J_0(t)$ be the accumulated current through the bond $\<-1,0\>$ (of particles in the lattice, and of mass in the continuum). We have the identity $J_0(t) = x_0(t)-x_0(0)$, and therefore the asymptotic behavior of the tagged particle is given by the asymptotic behavior of the current $J_0(t)$ for the process $\eta_t$.

A second identification, also known from the work of Rost and Vares, allows us to obtain the current $J_0(t)$ as a function of the empirical density associated to the process $\eta_t$. At least on a heuristic level, $J_0(t) = \sum_{i \geq 0} \{\eta_t(i) -\eta_0(i)\}$. When the number of particles is finite, this relation is simply an integrated form of the conservation of mass. When the number of particles is infinite, the sum is not absolutely summable. We will see that when the process $\eta_t$ is in {\em equilibrium}, the truncated sums $\sum_{0\leq i<n}  \{\eta_t(i) -\eta_0(i)\}$ form a Cauchy sequence, and the limiting variable is precisely the current $J_0(t)$. 

In this way, the asymptotic behavior of the tagged particle in single-file diffusions can be obtained in terms of the asymptotic behavior of the empirical process associated to the increments of the original process. This approach has been used by various authors \cite{LOV}, \cite{RV}, \cite{BS}, \cite{GJ} in order to obtain a central limit theorem for a tagged particle in single-file diffusions. 

The idea of relating the position of the tagged particle to the current of particles through the origin can also be accomplished for the original process, without considering the increment process. This approach was exploited in great generality in \cite{DGL}, where the authors obtain a functional central limit theorem for the tagged particle in a system of particles with collisions. Besides the collision rule, the evolution is independent. Considering different families of dynamics for the motion of one particle, they obtain in the limit any exponent $0<\gamma<2$ for the mean square displacement $E[x_0(t)^2]$ of the tagged particle. The main drawback of the approach of \cite{DGL} is that it does not generalize to systems of particles with stronger interaction.

For fairly general diffusive systems, the so-called {\em hydrodynamic limit} of the process $\eta_t$ is given by a non-linear heat equation of the form $\partial_t u = \partial_x ( D(u) \partial_x u)$, where $D(u)$ is the {\em bulk} diffusion coefficient, which is given by the Green-Kubo formula. We say that the process $\eta_t$ has an hydrodynamic limit  if the rescaled empirical process $n^{-1} \sum_{x \in \bb Z} \eta_{tn^2}(x)\delta_{x/n}(dx)$ converges in distribution to a deterministic limit of the form $u(t,x)dx$, where $u(t,x)$ is the solution of the hydrodynamic equation  $\partial_t u = \partial_x ( D(u) \partial_x u)$. From now on, we focus on lattice systems, so the empirical process represents the {\em density of particles} in the system. Notice the diffusive time-scaling. In particular, if we start the process with a fixed density of particles, at a macroscopic level the density of particles does not change. When the invariant measures of the process $\eta_t$ have short-range correlations, the spatial fluctuations of the density of particles are of Gaussian nature and they are given, in the macroscopic limit, by $\chi(\rho) \mc W$, where $\mc W$ is a white noise in $\bb R$, $\rho$ is the density of particles and the quantity $\chi(\rho)$ is the {\em static compressibility} of the system. By the fluctuation-dissipation relation, the density fluctuations around a fixed density $\rho$ evolves in a non-trivial way under a diffusive scaling, and they satisfy in the macroscopic limit the infinite-dimensional Ornstein-Uhlenbeck equation
\begin{equation}
\label{ec0}
d \mc Y_t = D(\rho) \Delta \mc Y_t dt + \sqrt{D(\rho)\chi(\rho)}\nabla d\mc W_t,
\end{equation}
where $\mc W_t$ is a space-time white noise. Going back to the representation of the current in terms of the empirical density of particles, we see that putting the Gaussian space scaling and the diffusive time scaling together, $n^{-1/2} J_0(tn^2)$ should converge to $\mc Y_t(H_0) - \mc Y_0(H_0)$, where $H_0(x) = \mathbf 1(x \geq 0)$ is the Heaviside function. Since the process is Gaussian, a simple scaling argument allow us to conclude that $J_0(t)$ approaches a normal distribution of variance $\Theta(\rho) t^{1/2}$ as $t \to \infty$, where $\Theta(\rho) = \chi(\rho)/\sqrt{D(\rho)}$. 

Recently (\cite{ASLL}, \cite{BS}), the following question has been posed. What happens with the tagged particle if each particle has its own, different diffusion constant? It turns out that \cite{GJ} under mild conditions, the behavior is the same as before: $E[x_0(t)^2]=\Theta(\rho)t^{1/2}$, but now the diffusion coefficient $D(\rho)$ is given by an homogenization formula. 

In this article we are interested in systems on which the asymptotic variance of the current (and therefore of  the tagged particle) is of order $t^{\gamma}$ for $\gamma \neq 1/2$. From a mathematical point of view, the universality of the $\gamma=1/2$ value is pleasant and satisfactory, but this is not a good fact from the point of view of modeling. There is also experimental and numeric evidence supporting values $\gamma <1/2$ for the variance growth of the tagged particle in some extreme situations. From a physical point of view, only values of $\gamma$ in $(0,1)$ are expected. In fact, if at a small time window we observe a positive increment on the current $J_0(t)$, this means that the density of particles at the right of the origin is larger than the density at left, and therefore we expect the current to have negative increments in the near future. This means that $J_0(t)$ has {\em negatively correlated increments}, restricting ourselves to concave functions for the variance of $J_0(t)$. 

Notice that the question about current fluctuations can be posed for any one-dimensional system, related or not to a single-file diffusion. Therefore, in this article we pose the question about the asymptotic behavior of  current and density fluctuations for general symmetric, one-dimensional particle systems. In particular, we propose two classes of models which will allow us to find a central limit theorem for $J_0(t)$ in the full range of scales $\gamma \in (0,1)$. We recall now that our heuristic derivation of the $\gamma =1/2$ law is very robust, we have only assumed that the static fluctuations are normal, that the hydrodynamic limit is diffusive, and that the fluctuation-dissipation relation holds. Following the same scheme, we will obtain a different value of $\gamma$ if the hydrodynamic limit of the process $\eta_t$ holds in a  {\em non-diffusive} scaling. If this is the case, we say that we are in prensence of an {\em anomalous diffusion}. In recent works, two symmetric models on which  anomalous diffusion occurs have been introduced. In \cite{FJL}, a simple exclusion process with degenerated bond disorder has been introduced. The simple exclusion process is just a system of simple, symmetric random walks in $\bb Z$, conditioned to never overlap. This is probably the most studied example of a single-file diffusion. Bond disorder is introduced in the following way. Let $\{\xi_x;x \in \bb Z\}$ be a sequence of i.i.d., positive  random variables. Assume that the common distribution is on the domain of attraction of an $\alpha$-stable law, $\alpha \in (0,1)$. For simplicity, take $\xi_x \geq 1$ for any $x \in \bb Z$. This does not alter the tail behavior of $\xi_x$, which turns out to be the relevant part.
We put a wall of size $\xi_x$ between sites $x$ and $x+1$, meaning that each time a particle tries to jump from $x$ to $x+1$ or viceversa (and if the jump is allowed by the exclusion rule), the jump is accomplished with probability $\xi_x^{-1}$. Due to the heavy tails of $\xi_x$, the dynamics is dramatically slowed down. In fact, the correct time scaling is $n^{1+1/\alpha}$, which is always bigger than $n^2$, and the hydrodynamic equation is $\partial_t u = \partial_x \partial_W u$, where $W$ is an $\alpha$-stable subordinator corresponding to the scaling limit of the walls and $\partial_W$ (also denoted by $d/dW$ in the sequel) denotes the inverse of the Stieltjes integral with respect to $W$. As we can see, the randomness of the environment is so strong that it survives in the limit, and even the hydrodynamic equation depends on the corresponding realization of the environment. This scaling is robust in the sense that stronger interaction between particles does not lead to a different time scaling \cite{FL}.

Looking back to the heuristic formula for $\gamma$, we obtain $\gamma = \alpha/(1+\alpha)<1/2$ for any $\alpha \in (0,1)$, so the current in this model should satisfy $E[J_0(t)^2] \sim t^{\alpha/(1+\alpha)}$ for large $t$. Instead of considering the simple exclusion process with bond disorder, in this article we study the {\em zero-range process} with degenerated bond disorder. In this way we emphasize that the subdiffusive behavior holds regardless of the details of the local interaction. For the simple exclusion process with variable diffusion coefficient defined in \cite{GJ}, the increment process is precisely a zero-range process with bond disorder. Therefore, a central limit theorem for the current through the origin leads to a central limit theorem for the tagged particle in this last model, which falls into the category of single-file diffusions.

A second example of a symmetric system with non-diffusive hydrodynamic limit is the zero-range process with long jumps \cite{Jar}. In this system particles interact between them only when they share positions, and the jump probability of the underlying random walk satisfies a power law: $p(z) =c/|z|^{1+\alpha}$, $\alpha \in (0,2)$. In this case the correct scaling limit is $n^\alpha$, which is always smaller than $n^2$, and the hydrodynamic equation is of the form $\partial_t u = \Delta_\alpha \varphi(u)$, where $\Delta_\alpha = -(-\Delta)^{\alpha/2}$ is the fractional Laplacian and $\varphi(\cdot)$is a function encoding the interaction between particles. In this case, the heuristic formula gives $\gamma = 1/\alpha$, which is always bigger than $1/2$. Notice that for $\alpha \leq 1$, this formula gives $\gamma \geq 1$. It turns out that for $\alpha \leq 1$, the current through the origin is not well defined, since in any time window $[t,t+h]$ there is an infinite number of particles crossing from one side of the origin to the other in that time window. 

Since the particle jumps are not restricted to nearest-neighbors, it is not possible to find a single-file diffusion for which the increment process falls into this class. Therefore, for any $\gamma \in (0,1)$, we have a model for which the current of particles should satisfy $E[J_0(t)^2] \sim t^\gamma$, but for the tagged particle problem we have a model for which $E[x_0(t)^2] \sim t^\gamma$ only when $\gamma \leq 1/2$. 

Although the heuristic plan looks simple and it has been accomplished in the diffusive case for many examples, anomalous diffusive behavior poses difficulties that are absent for diffusive systems. The main technical difficulty relies on establishing the fluctuation-dissipation theorem for the process $\eta_t$. In the superdiffusive case, the main obstacle is that the fractional Laplacian $\Delta_\alpha$ does not leave invariant the Schwartz space $\mc S(\bb R)$ of test functions, and therefore the classical construction of generalized Ornstein-Uhlenbeck processes due to Holley and Stroock \cite{HS} does not apply. Solutions of (\ref{ec0}) can be constructed using the formalism of Gaussian process for fairly general driving, non-positive operators $\mc L$ \cite{AC}. However, this construction is not suitable for proving convergence when (as it is the case here) it is not easy to show that the limiting process has Gaussian distributions regardless of initial conditions. Section \ref{s1} is basically a recall of  \cite{DFG} and \cite{DG}, where powerful methods have been developed to prove such convergence theorem, very much in the spirit of \cite{HS}. We include this section with no new results for the reader's convenience, since we are not aware of previous results applying these ideas to lattice systems. In particular a notion for uniqueness of equation (\ref{ec0}) is stated, and the notion of intermediate spaces is introduced. 

In the subdiffusive case, even the definition of the corresponding Ornstein-Uhlenbeck process poses a challenge. In fact, besides constant functions, there are no smooth functions on the domain of the operator $\mc L_W = \partial_x \partial_W$. Moreover, for any two realizations of the subordinator $W$, the domains of the corresponding operators $\mc L_W$ have in common only constant functions. The new material on this article starts at Section \ref{s2}. In Section \ref{s2} we construct a nuclear Fr\'echet space $\mc F = \mc F_W$ which will serve as a test space in order to define the generalized Ornstein-Uhlenbeck process associated to the operator $\mc L_W$. Our key input is a compactness result for weighted-Sobolev spaces, very much in the spirit of the definition of the usual weighted-Sobolev spaces in $\bb R$. Our construction works for any increasing, unbounded function $W$, and could be of independent interest. In Section \ref{s3} we give detailed definitions of the zero-range process wiht random environment and with long jumps. We also state our main results concerning the asymptotic behavior of the density and the current of particles. In Section \ref{s4} we obtain the fluctuations of the density in the superdiffusive case. In Section \ref{s5} we obtain the fluctuations of the density in the subdiffusive case. A key intermediate result is the so-called {\em energy estimate}, which roughly says that the space-time fluctuations of a given function can be estimated by the Dirichlet form associated to the underlying random walk. This result holds true for any reversible system, regardless of the super or subdiffusive behavior of the system. The universality of this estimate is more evident in Section \ref{s7}, where we prove fluctuations for the current of particles through the origin in both super and subdiffusive cases. We finish this article in Section \ref{s7} by obtaining the fluctuations of a tagged particle in the simple exclusion process with variable diffusion coefficient, as a direct consequence of the results in Section \ref{s6}. We point out that all our results in the subdiffusive case, applies for any process $W(x)$, stochastic or not, such that $\lim_{n \to \pm \infty} W(x) = \pm \infty$ and such that the environment has a version converging almost surely to $W(x)$.

\section{Generalized Ornstein-Uhlenbeck processes}
\label{s1}

In this section we give a precise definition to what we mean by a generalized, or infinite-dimensional, Ornstein-Uhlenbeck process and we state some conditions for uniqueness of such processes. All the material in this section has been taken from \cite{DFG} and \cite{DG}. The interested reader can find a more detailed exposition and further applications in those articles.

\begin{remark}
Throughout this article, we use the denomination ``Proposition'' for result that have been proved elsewhere. We reserve the denomination ``Theorem'' for original results.
\end{remark}

\subsection{Preliminary definitions}

Let $\mc L^2(\bb R^d)$ be the Hilbert space of square integrable functions $\varphi: \bb R^d \to \bb R$. Let $\mc L: D(\mc L) \subseteq \mc L^2(\bb R^d) \to \mc L^2(\bb R^d)$ be the generator of a strongly continuous contraction semigroup $\{S_t; t \geq 0\}$ in $\mc L^2(\bb R^d)$. Let $\{||\cdot||_n\}_n$ be an increasing family of (not necessarily finite) norms in $\mc L^2(\bb R^d)$ with a common kernel $\mc F_0$ such that $||\varphi||_0^2 = \int \varphi(x)^2dx$. Let $\mc F \subseteq \mc L^2(\bb R^d)$ be the Fr\'echet space generated by $\{||\cdot||_n\}_n$, that is, the completion of $\mc F_0$ under the metric 
\[
 d(\varphi,\psi) = \sum_{n \geq 0} 2^{-n} \big(||\varphi -\psi||_n \wedge 1\big).
\]

Let us denote by $\mc F'$ the topological dual of $\mc F$. We can construct $\mc F'$ in such a way that $\mc F \subseteq \mc L^2(\bb R^d) \subseteq \mc F'$ and such that the inner product $\<\cdot,\cdot\>$ in $\mc L^2(\bb R^d)$ restricted to $\mc F_0 \times \mc F_0$ can be continuously extended to a continuous bilinear form $\<\cdot,\cdot\> : \mc F \times \mc F' \to \bb R$. We assume that the space $\mc F$ is {\em nuclear}, that is, for any $n \geq 0$ there exists $m>n$ such that any $||\cdot||_m$-bounded set is a $||\cdot||_n$-compact set. 
In what follows we will always consider a family of norms $||\cdot||_n$ for which the set $\{\varphi \in \mc L^2(\bb R^d); ||\varphi||_n<+\infty\}$ is a Hilbert space under $||\cdot||_n$, although this point is not essential.
Our objective is to describe some conditions under which existence and uniqueness of solutions can be established for the stochastic equation
\begin{equation}
\label{ec1}
 d \mc Y_t = \mc L^* \mc Y_t dt + d\mc Z_t,
\end{equation}
where $\mc Z_t$ is a given semimartingale in $\mc F'$ and $\mc L$ is a given, maybe unbounded, operator in $\mc L^2(\bb R^d)$. The canonical example of a nuclear, Fr\'echet space is the Schwartz space $\mc S(\bb R^d)$ of test functions in $\bb R^d$. In this case,
\[
 ||\varphi||_n = \bigg\{ \sum_{|k| \leq n} \int_{\bb R^d} (1+x^2)^n\big(\partial^k\varphi(x)\big)^2 dx \bigg\}^{1/2},
\]
where $k$ denotes a multi-index $(k_1,\dots,k_d)$ with  $|k|=k_1+\cdots+k_d$ and $\partial^k = \prod_i \partial_{x_i}^{k_i}$. The dual of $\mc S(\bb R^d)$ is the space $\mc S'(\bb R^d)$ of tempered distributions and a common kernel for each norm $||\cdot||_n$ is the set $\mc C_c^\infty(\bb R^d)$ of infinitely differentiable functions in $\bb R^d$ of compact support. 

For a given topological space $\mc E$, we denote by $\mc D([0,T],\mc E)$ the space of c\`adl\`ag trajectories in $\mc E$. For simplicity, we consider a finite time interval $[0,T]$; results for $[0,\infty)$ will follow from standard extension arguments. For $a<b$, we denote by $\mc C^\infty(a,b)$ the space of $\mc C^\infty$ functions in $[a,b]$ with support contained in $(a,b)$. The space $\mc C^\infty(a,b)$ is a nuclear Fr\'echet space with respect to the topology of uniform convergence on compacts of the function and its derivatives of any order. Notice that functions in $\mc C^\infty(a,b)$ vanish at $t=a,b$. 

Let us denote by $\mc F_{0,T}$ the tensor product $\mc F \otimes \mc C^\infty(0,T)$. For a clear exposition about tensor products, see \cite{Tre}. This space is also a Fr\'echet space. We denote by $\mc F_{0,T}'$ the topological dual of $\mc F_{0,T}$. Now we state some technical lemmas which will be useful in the identification of solutions of (\ref{ec1}).

\begin{lemma}
 For any trajectory $x_\cdot$ in $\mc D([0,T],\mc F')$, let $\tilde x \in \mc F'_{0,T}$ be defined by
 \[
  \<\tilde x, \psi\> = \int_0^T \<x_t, \psi_t \> dt
 \]
 for any $\psi \in \mc F_{0,T}$. Then the mapping $x_\cdot \mapsto \tilde x$ is a continuous mapping from the space $\mc D([0,T],\mc F')$ into $\mc F'_{0,T}$.
\end{lemma}

\begin{lemma}
\label{l2}
 Let $x_\cdot$ be a process in $\mc D([0,T],\mc F')$, $a.s.$ continuous at $t=T$. Then the distributions of $x_\cdot$, $\tilde x$ determine each other.
\end{lemma}

\begin{definition}
We say that a family $\{X_f; f \in \mc F\}$ of integrable random variables is a linear random functional if $X_f$ is linear and continuous as a function of $f$. 
\end{definition}

\begin{lemma}
\label{l2.1}
 Let $X_f$ be a linear random functional. Then, there exists a unique random variable $X$ in $\mc F'$ such that $\<X,f\> = X_f$ $a.s.$ for every $f \in \mc F$. 
\end{lemma}

\subsection{Weak formulation}

When the operators $\mc L$ and $S_t$ have the good taste of leaving the space $\mc F$ invariant, there is a very intuitive notion of solutions for (\ref{ec1}). This is the case, for example, when $\mc L =\Delta$. We start doing some formal manipulations. Here and below, the initial condition $\mc Y_0$ will be a random variable with values in $\mc F'$. Apply (\ref{ec1}) to a test function $\varphi$, multiply by another test function $f \in \mc C^\infty(-\delta,T)$, integrate the equation over time and perform an integration by parts to obtain
\[
 \int_0^T \<\mc Y_t, \varphi f'(t)+ f(t)\mc L \varphi\> dt
 	= \< \mc Y_0, \varphi f(0)\> + \int_0^T \<\mc Z_t, \varphi f'(t)\> dt.
\]

Since we want to capture the initial distribution, we will extend the space $\mc F_{0,T}$ a little bit. Notice that a function $f \in \mc F_{0,T}$ can be thought as a trajectory in $\mc F$ which vanishes at $t = 0, T$. Take any $\delta >0$ and define $\mc F_{-\delta,T}$ as the tensor product $\mc F \otimes \mc C^\infty(-\delta,T)$. We define $\mc F_T$ as the set of trajectories in $\mc F_{-\delta,T}$ restricted to the interval $[0,T]$. These trajectories always vanish at $t=T$, and of course $\mc F_{0,T} \subseteq \mc F_T$. By linearity and an approximation procedure, we can extend the previous identity to test functions $\psi \in \mc F_T$. For $\psi \in \mc F_T$, this identity reads
\begin{equation}
\label{ec2}
 \int_0^T \< \mc Y_t, \frac{\partial}{\partial t} \psi_t+\mc L \psi_t\> dt 
 	=\<\mc Y_0,\psi_0\> + \int_0^T \< \mc Z_t, \frac{\partial}{\partial t} \psi_t\> dt.
\end{equation}

\begin{definition}
\label{d1}
 We say that a process $\mc Y_\cdot$ in $\mc D([0,T],\mc F')$ is a solution of (\ref{ec1}) if $\mc L: \mc F \to \mc F$ is continuous and (\ref{ec2}) holds for any $\psi \in \mc F_T$.
\end{definition}

Let us recall now the variation of parameters method to solve linear evolution equations. The equation $d \mc{Y}_t = \mc  L^* \mc Y_t$ has as a solution the process $\mc Y_t = S_t^* \mc Y_0$, which means that $\mc Y_t$ is defined via the relation $\<\mc Y_t, \varphi\> = \< \mc Y_0, S_t \varphi\>$. Notice that $\mc Y_t$ is well defined in $\mc D([0,T],\mc F')$ if $S_t: \mc F \to \mc F$ is continuous. The variation of parameters method suggests to search for a solution to (\ref{ec2}) of the form $\mc Y_t = S_t^* \mc X_t$, where $\mc X_t$ is  a semimartingale satisfying $\mc X_0= \mc Y_0$. In that case, $\mc Y_t$ formally satisfies (\ref{ec2}) if $d\mc X_t = S_{-t}^* d \mc Z_t$. Since $S_{-t}^*$ is not well defined, the usual trick is to multiply this expression by $S_{t'}^*$ to obtain a well defined relation for $0 \leq t \leq t'$: $S_{t'}^* d\mc X_t = S_{t'-t}^* d\mc Z_t$. Integrating this expression, we obtain that $S_t^* \mc X_t = S_t^* \mc Y_0 +\int_0^t S_{t-s}^* d \mc Z_t$.

In terms of test functions, this expression gives
\[
 \<\mc Y_t, \varphi\> = \<\mc Y_0, S_t \varphi\> + \int_0^t \< d \mc Z_s, S_{t-s} \varphi\>.
\]
 
For a test function of the form $\varphi(x) f(t)$ with $\varphi \in \mc F$ and $f \in \mc C^\infty(0,T)$, the same manipulations yield
\[
 \int_0^T \<\mc Y_t, \varphi f(t)\> dt = \< \mc Y_0, \int_0^T f(t) S_t \varphi dt\> + \int_0^T dt \int_0^t \< d \mc Z_s, S_{t-s} \varphi\> f(t). 
\]

The last term can be written as
\begin{align*}
 \int_0^T \< d \mc Z_s, \int_s^T f(t) S_{t-s} \varphi dt\> 
 	&= - \int_0^T \< \mc Z_s, -\varphi f(s) \> - \int_s^T f(t) S_{t-s} \mc L \varphi dt \> ds\\
	& = -\int_0^T \<\mc Z_s, \int_s^T f'(t) S_{t-s} \varphi dt\> ds.   
\end{align*}

As before, using linearity and an approximation procedure we can extend this identity for arbitrary test functions $\psi \in \mc F_T$:
\begin{equation}
 \label{ec3}
 \int_0^T \<\mc Y_t, \psi_t\> dt = \< \mc Y_0, \int_0^T S_t \psi_t dt\> - \int_0^T \< \mc Z_s, \int_s^T S_{t-s} \frac{\partial}{\partial t} \psi_t dt\> ds.
\end{equation}

\begin{definition}
\label{d2}
 A process $\mc Y_\cdot$ in $\mc D([0,T],\mc F')$ is said to be an evolution solution of (\ref{ec1}) if $S_t: \mc F \to \mc F$ is continuous, for $\varphi \in \mc F$ the curve $t \mapsto \<\mc Y_t,\varphi\>$ is continuous in $\mc F$, and (\ref{ec3}) is satisfied for any $\psi \in \mc F_T$.
\end{definition}

Notice that by Lemma \ref{l2}, relation (\ref{ec3}) determines the distribution of $\mc Y_\cdot$.
Now we explain on which sense we have a unique solution of (\ref{ec1}):

\begin{proposition}
\label{p1}
 Let $\mc Y_0$ be a distribution in $\mc F'$. There exists a unique solution $\mc Y_\cdot$ of (\ref{ec1}) with initial condition $\mc Y_0$ if the conditions on $\mc L$ and $S_t$ stated in Definitions \ref{d1}, \ref{d2} are fulfilled. This solution is given by the evolution solution defined in Definition \ref{d2}.
\end{proposition}

Uniqueness follows from (\ref{ec3}) and Lemma \ref{l2}. Existence follows from standard methods for evolution equations. We call the process $\mc Y_\cdot$ the {\em generalized Ornstein-Uhlenbeck process} of characteristics $(\mc L,\mc Z_t)$ and initial condition $\mc Y_0$. The most well known example for which this Proposition applies is the operator $\mc L= \Delta$. In that case, it is well known that $\mc F =\mc S(\bb R^d)$ is left invariant by $\mc L$ and by $S_t$. However, an important example, relevant for our purposes, which falls out of this setting is the operator $\mc L = -(-\Delta)^{\alpha/2}$, $\alpha \in (0,2)$, i.e., when $\mc L$ is the {\em fractional Laplacian} in $\mc R^d$. The fractional Laplacian is an integral operator which can be written in the form
\[
 \mc L \varphi(x) = c_\alpha \int_{\bb R^d} \frac{dy}{|y|^{d+\alpha}} \big\{\varphi(x+y)-\varphi(x)\big\}
\]
for a properly chosen constant $c_\alpha>0$. Due to the long-range integration, for any positive function $\varphi \in \mc S(\bb R^d)$, the function $\mc L \mc \varphi \notin \mc S(\bb R^d)$: decay at infinity fails, as it is easily shown taking Fourier transforms. Therefore, for the case on which $\mc L = -(-\Delta)^{\alpha/2}$ a new interpretation of (\ref{ec1}) is needed. This is the content of the following section.

\subsection{The intermediate spaces}

Throughout this section, we take $\mc F = \mc S(\bb R^d)$, and therefore $\mc F' = \mc S'(\bb R^d)$. For $p>0$ and $\varphi \in \mc C(\bb R^d)$, the space of continuous functions in $\bb R^d$, define
\[
 ||\varphi||_{p,\infty} = \sup_{x \in \bb R^d} \big|\varphi(x)\big|(1+x^2)^p
\]
and $\mc C_p = \{\varphi \in \mc C(\bb R^d); ||\varphi||_{p,\infty} < +\infty\}$. Define also
\[
 \mc C_{p,0} = \{ \varphi \in \mc C(\bb R^d); \lim_{|x| \to \infty} \varphi(x)(1+x^2)^p =0\}.
\]

Clearly, $\mc C_p$, $\mc C_{p,0}$ are Banach spaces with respect to the norm $||\cdot||_{p,\infty}$. We also have the continuous embeddings $\mc F \hookrightarrow \mc C_{p,\infty}$ for $p >0$ and $\mc C_{p,0} \hookrightarrow \mc L^2(\bb R^d)$ for $p> d/2$. Let $\mc C_{p,\infty}'$ be the topological dual of $\mc C_{p,\infty}$. We have the chain of inclusions
\[
 \mc F \hookrightarrow \mc C_{p,0} \hookrightarrow \mc L^2(\bb R^d) \hookrightarrow \mc C_{p,0}' \hookrightarrow \mc F'.
\]

Denote by $||\cdot||_{-p,\infty}$ the dual norm in $\mc C_{p,\infty}'$. Let $\mc M_p^+$ be the set of positive, Radon measures $\mu$ in $\bb R^d$ such that $\int (1+x^2)^{-p} \mu(dx) <+\infty$. We give to $\mc M_p^+$ the $p$-vague topology, which is the weakest topology that makes the mappings $\varphi \mapsto \<\mu,\varphi\> =: \int \varphi d\mu$ continuous for every $\varphi \in \mc C(\bb R^d)$ of compact support and for $\varphi(x) = (1+x^2)^{-p}$ as well. Notice that for $\mu \in \mc M_p^+$, $||\mu||_{-p,\infty} = \int (1+x^2)^{-p} \mu(dx)$. Notice also that Lebesgue measure is in $\mc M_p^+$ for $p>d/2$. Define $\Delta_\alpha =-(-\Delta)^{\alpha/2}$ and let $S_t$ be the semigroup generated by $\Delta_\alpha$. The following propositions can be proved taking the representation of $\Delta_\alpha$ in Fourier space and the representation of $S_t$ in Fourier-Laplace space.

\begin{proposition}
 For $t \geq 0$ and $d/2 < p<(d+\alpha)/2$, $S_t$ is a bounded linear operator from $\mc C_p$ into itself, and also from $\mc M_p^+$ into itself.
\end{proposition}

\begin{proposition}
 Let $\varphi \in \mc C_p$ be such that the limit $\lim_{|x| \to \infty} \varphi(x)(1+x^2)$ exists. Then $t \mapsto S_t \varphi$ is a continuous trajectory in $\mc C_p$. For any $\mu \in \mc M_p^+$, $t \mapsto S_t \mu$ is a $p$-vaguely continuous trajectory in $\mc M_p^+$. Moreover, there is a positive constant $C_T$ such that
 \[
  ||S_t \varphi||_{p,\infty} \leq C_T ||\varphi||_{p,\infty}
 \]
 for any $\varphi \in \mc C_p$ and any $0 \leq t \leq T$.
\end{proposition}

\begin{proposition}
 The mapping $\varphi \in \mc F \mapsto \Delta_\alpha \varphi \in \mc C_{p,0}$ is continuous. 
\end{proposition}

The idea now is to define a generalized version of what we mean by a solution of (\ref{ec1}). Notice that {\em a priori} expresions of the type $\<\mc Y_t,\Delta_\alpha \varphi\>$ are not well defined, since $\Delta_\alpha \varphi$ does not belong to $\mc F$. In the following two definitions, $\mc F$ denotes an arbitrary Fr\'echet space.

\begin{definition}
\label{d3}
 We say that a process $\mc Y_\cdot$ in $\mc D([0,T],\mc F')$ is a generalized solution of the (\ref{ec1}) if there exists a Banach space $\mc V$ of functions in $\bb R^d$ such that 
 \begin{itemize}
  \item[i)] $\mc F \hookrightarrow \mc V \hookrightarrow \mc L^2(\bb R^d)$
  \item[ii)] The linear mapping $\mc L : \mc F \hookrightarrow \mc V$ is continuous
  \item[iii)] $\int_0^T \< \mc Y_t, \mc L \psi_t\> dt$ is well defined for any $\psi_\cdot \in (D(\mc L) \cap \mc V) \otimes \mc C^\infty (0,T)$
  \item[iv)] For any $\psi_\cdot$ as in $iii)$, the identity (\ref{ec2}) holds $a.s.$
 \end{itemize}
\end{definition}

\begin{definition}
\label{d4}
 A process $\mc Y_\cdot$ in $\mc D([0,T],\mc F')$ is a generalized evolution solution of (\ref{ec1}) if there exists a Banach space $\mc V$ such that
 \begin{itemize}
  \item[i)] $\mc F \hookrightarrow \mc V \hookrightarrow \mc L^2(\bb R^d)$
  \item[ii)] $S_t: \mc F \to \mc V$ is continuous for any $t \in [0,T]$
  \item[iii)] $t \mapsto S_t \varphi$ is a continuous trajectory in $\mc V$ for any $\varphi \in \mc F$
  \item[iv)] The right-hand side of (\ref{ec3}) is well defined for any $\psi_\cdot \in \mc V \otimes \mc C^\infty(0,T)$
  \item[v)] For any $\psi_\cdot$ as in $iv)$, the identity (\ref{ec3}) holds $a.s.$
 \end{itemize}

\end{definition}

\begin{proposition}
\label{p2}
 Under the conditions of Definitions \ref{d3}, \ref{d4}, for any initial distribution $\mc Y_0$ in $\mc F'$ there exists a unique generalized solution of (\ref{ec1}), which is given by the generalized evolution solution defined in Definition \ref{d4}.
\end{proposition}

Notice that in particular, the stochastic equation $d \mc Y_t = \Delta_\alpha^* \mc Y_t dt + d\mc Z_t$ has a unique solution for any initial distribution $\mc Y_0$ in $\mc F'$.

\section{Generalized differential operators}
\label{s2}
\subsection{The derivative $d/dW$}

In Section \ref{s1} we have discussed existence and uniqueness of generalized Ornstein-Uhlenbeck processes defined in dual Fr\'echet spaces. Up to this point, the examples discussed so far (corresponding to the fractional Laplacians $\Delta_\alpha$) do not require to take an abstract Fr\'echet space, since the only space considered was the space $\mc S'(\bb R^d)$ of tempered distributions. In this section we give an example where ad-hoc nuclear spaces need to be constructed. 

Let $W: \bb R \to \bb R$ be a strictly increasing function. By convention, we assume that $W$ is c\`adl\`ag and that $W(0)=0$. For two given functions $f,g: \bb R \to \bb R$, we say that $g = df/dW$ if the identity
\[
 f(x) = f(0+) +\int_0^x g(y) W(dy)
\]
holds for any $x \in \bb R$, where the integral is understood as a Stieltjes integral, and $\int_0^x$ means integration over the interval $(0,x]$. In the same spirit, we say that $g = d/dx df/dW$ if
\[
 f(x) = f(0) + \frac{df}{dW}(0)\big\{W(x)-W(0)\big\} +\int_0^x W(dz) \int_0^z dy g(y)
\]
for any $x \in \bb R$. We will denote the operator $d/dx d/dW$ by $\mc L_W$. Notice that when the function $W$ is differentiable and $W'(x) \neq 0$ for any $x$ (that, $W$ is a diffeomorphism), $\mc L_W = \partial_x(W'(x)^{-1}\partial_x)$. We say that a function $f$ is $W$-differentiable if there exists a function $g: \bb R \to \bb R$ such that $g = df/dW$. When $W$ is a diffeomorphism, this notion reduces to usual differentiability.

Now we want to define $\mc L_W$ as an unbounded operator in $\mc L^2(\bb R)$. We say that a function $f \in \mc L^2(\bb R)$ belongs to $D(\mc L_W)$, the {\em domain} of $\mc L_W$, if there exists constants $a, b \in \bb R$ and a function $g \in \mc L^2(\bb R)$ such that
\begin{equation}
\label{ec4}
 f(x) = a+b\big(W(x)-W(0)\big) + \int_0^x W(dz) \int_0^z dy g(y)
\end{equation}
for almost every $x \in \bb R$. In that in this case, $f$ has a c\`adl\`ag version for which (\ref{ec4}) holds for any $x \in \bb R$. Notice also that the double integral is well defined for any $x$, since $g \in \mc L^1_{loc}(\bb R)$ and therefore the integrand of the second iterated integral is continuous. From now on, we assume that any $f \in D(\mc L_W)$ is c\`adl\`ag. If $f \in D(\mc L_W)$, we define $\mc L_W f =g$. Evidently, the operator $\mc L_W: D(\mc L_W) \to \mc L^2(\bb R)$ is a linear, unbounded operator. What it is not so evident, is that $\mc L_W$ is densely defined, symmetric and non-positive (remember that $\mc L_W$ is densely defined if $D(\mc L_W)$ is dense in $\mc L^2(\bb R^d)$). We state these properties as lemmas.

\begin{lemma}
 The operator $\mc L_W$ is densely defined.
 \label{l3}
\end{lemma}

\begin{proof}
 Let $\zeta: \bb R\to [0,1]$ be a $\mc C^\infty$ function such that $\zeta(x) =0$ if $x \notin (0,1)$ and $\zeta(x)>0$ if $x \in (0,1)$. Fix $a<b$. Take $0< \delta < (b-a)/2$ and define $h: \bb R \to \bb R$ by
 \begin{itemize}
  \item[i)] for $x\leq a$, $h(x)=0$ 
  \item[ii)] for $a < x \leq (a+b)/2$,
  \[
   h(x) = c(a,\delta)^{-1} \int_a^x \zeta \big( ( y-a)/\delta \big) W(dy),
  \]
where $c(a,\delta)= \int_a^{a+\delta} \zeta \big( ( y-a)/\delta \big) W(dy)$
  \item[iii)] for $(a+b)/2 \leq x <b$,
  \[
   h(x) = h\big( (a+b)/2\big) - c(b,\delta)^{-1} \int_{\frac{a+b}{2}}^x \zeta \big((b-y)/\delta \big) W(dy),
  \]
where $c(b,\delta)= \int_{b-\delta}^b \zeta \big((b-y)/\delta \big) W(dy)$
  \item[iv)] $h(x)=0$ for $x \geq b$.
 \end{itemize}
 
 By construction, $h \in D(\mc L_W)$ and when $\delta \to 0$, $h$ converges in $\mc L^2(\bb R)$ to $\mathbf 1(x \in [a,b])$, the indicator function of the interval $[a,b]$. Considering the vectorial space generated by the functions $h$, we conclude that $D(\mc L_W)$ is dense in $\mc L^2(\bb R)$. 
\end{proof}

\begin{remark}
\label{r1}
Starting the construction with a function $\zeta$ in $D(\mc L_W)$ instead of a smooth function, we obtain a kernel for $(\mc L_W)^2$. A simple iteration of this construction allows us to prove that in fact the operator $(\mc L_W)^n$ is densely defined for any $n >0$. 
\end{remark}

\begin{lemma}
 The operator $\mc L_W$ is symmetric and non-positive.
\end{lemma}

\begin{proof}
 Take $f,g \in D(\mc L_W)$, of compact support. By the previous lemma these functions are dense in $\mc L^2(\bb R)$. Then,
 \begin{align*}
  \int_{\bb R} f \mc L_W g dx 
  	& = \lim_{M \to \infty} \Big\{ f \frac{dg}{dW}\Big|_{x=-M}^{x=M} - \int_{-M}^M \frac{dg}{dW} df \Big\}\\
	& = - \int_{\bb R} \frac{df}{dW} \frac{dg}{dW} dW,
 \end{align*}
 where we have used the integration by parts formula for Stieltjes integrals and the fact that $df/dW$ and $dg/dW$ are well defined and have bounded support. By an approximation argument, we conclude the $\int f \mc L_W g dx = \int g \mc L_W f dx$ for any $f, g  \in D(\mc L_W)$, which proves that $\mc L_W$ is symmetric. Taking $f=g$, we see that $\int f \mc L_W f dx = - \int (df/dW)^2 dW$ and therefore $\mc L_W$ is non-positive.
\end{proof}

\subsection{The Sobolev spaces associated to $\mc L_W$}
In this section we construct a sequence of nested Hilbert spaces $\mc H_{n+1} \subseteq \mc H_n \subseteq \cdots \subseteq \mc L^2(\bb R)$ and we define $\mc F$ as the intersection of such spaces. We will prove that $\mc F$ is a nuclear Fr\'echet space on which $\mc L_W$ is continuous and therefore we will show that the theory of generalized Ornstein-Uhlenbeck process of Section \ref{s1} applies to $\mc F$ and $\mc L_W$. The idea is simple. We start recalling the spirit behind the definition of the usual Sobolev spaces in $\bb R$. In finite volume, the Laplacian operator with boundary conditions of Dirichlet type satisfies Poincar\'e inequality. From Poincar\'e inequality, we can prove that the Laplacian $\Delta$ has a compact resolvent. In particular, the sequence of Sobolev spaces associated to the norms $||\varphi||_n = \<\varphi, (-\Delta)^n \varphi\>^{1/2}$ is nuclear, and we also know that its intersection contains all the infinitely differentiable functions of compact support. However, in infinite volume the operator $\Delta$ does not have a compact resolvent, and another definition is needed. We already know that the Sobolev norms in infinite volume include a polynomial weight. One way to understand this polynomial weight is to look at the {\em Helmholtz operator} $-\Delta + x^2$. This operator has a compact resolvent, and its eigenvectors are the Hermite functions. The vectorial space finitely generated by the Hermite functions is the set of functions of the form $p(x) e^{-x^2}$, where $p(x)$ is a polynomial. A possible definition for the Schwartz space $\mc S(\bb R)$ is the following. Take $||\varphi||_n = \<\varphi, (-\Delta +x^2)^n \varphi\>^{1/2}$, $\mc F_0 = \{p(x)e^{-x^2}; p(\cdot) \text{is a polynomial }\}$ and define $\mc S(\bb R)$ as the closure of $\mc F_0$ under the topology generated by the norms $||\cdot||_n$. 

We will proceed in an analogous way for $\mc L_W$. Define $\mc A = -\mc L_W + |W(x)|$ (here we understand $|W(x)|$ as a multiplication operator). In the following series of lemmas, we will construct a nuclear space $\mc F$ on which $\mc L_W$ is continuous. We assume that $W$ is continuous at $x=0$. This assumption is not restrictive, since $\mc L_W$ remains unchanged by adding a constant to $W$ and $W$ has at most a numerable set of discontinuities. Therefore, a simple shift of the origin makes $W$ continuous at $x=0$. 

We start with a very simple observation.

\begin{lemma}
 The operators $\mc L_W$ and $|W(x)|$ have a common kernel $\mc D$.
\end{lemma}

\begin{proof}
 It is enough to take $\mc D$ as the set of functions constructed in Lemma \ref{l3}. In fact, by construction this set is a kernel for $\mc L_W$, and it is well known that the set of functions of compact support is a kernel for any multiplication operator on which the multiplication function is locally bounded.
\end{proof}
Notice that from this proposition we conclude that the operator $\mc A$ is densely defined, a fact which is not clear {\em a priori}. The same arguments also apply for the operator $\mc A^n$, $n \geq 1$. 
For $\varphi \in D(\mc A^n)$ define
\[
 ||\varphi||_n = \Big \{ \sum_{k=0}^n \<\varphi, \mc A^k \varphi\>\Big\}^{1/2}.
\]

Define now the {\em Sobolev space} $\mc H_n = \mc H_n(W)$ as the completion of $D(\mc A^n)$ under the norm $||\cdot||_n$. Notice that $\mc H_{n+1} \subseteq \mc H_n \subseteq \cdots \subseteq \mc L^2(\bb R)$ and each space $\mc H_n$ is a Hilbert space with inner product obtained from $||\cdot||_n$ by polarization. 

Before stating the next lemma, we need a simple definition. We say that a function $f: \bb R \to \bb R$ is $W$-H\"older continuous of exponent $\beta>0$ if there is a constant $c>0$ such that $|f(x)-f(y)| \leq c|W(x)-W(y)|^\beta$ for any $x, y \in \bb R$. 

\begin{lemma}
 Any function $\varphi \in D(\mc A)$ is $W$-H\"older continuous of exponent $1/2$. 
 \label{l4}
\end{lemma}

\begin{proof}
 Notice that 
 \[
  ||\varphi||_1^2 = \int \Big(\frac{d\varphi}{dW}\Big)^2 dW + \int |W(x)| \varphi(x)^2 dx.
 \]

 It is enough to see that 
 \begin{align*}
  \big|\varphi(y)-\varphi(x)\big|^2 
  	&= \Big(\int_x^y \frac{d\varphi}{dW} dW\Big)^2 
  	\leq \Big| W(y)-W(x)\Big|\int_x^y \Big(\frac{d\varphi}{dW}\Big)^2 dW\\
	&\leq ||\varphi||_1^2 \big|W(y)-W(x)\big|,
 \end{align*}
 which proves the lemma.
\end{proof}

\begin{lemma}
\label{l5}
 Let $\{\varphi_n\}_n$ be a sequence in $\mc L^2(\bb R)$ such that $||\varphi_n||_1 \leq 1$ for every $n$. Then, there are a subsequence $n'$ and a function $\varphi \in \mc L^2(\bb R)$ such that $||\varphi_{n'} -\varphi||_0 \to0$ as $n \to \infty$. In particular, the embedding $\mc H_1 \hookrightarrow \mc L^2(\bb R)$ is compact.
\end{lemma}

\begin{proof}
 By Lemma \ref{l4}, the sequence $\{\varphi_n\}_n$ is $W$-H\"older continuous of exponent $1/2$. Consider in $\bb R$ the distance $d_W(x,y) = |W(y)-W(x)|$. The metric space $(\bb R, d_W)$ is not complete, but its completion can be obtained by ``splitting in two'' each point of discontinuity of $W$. More precisely, let $\{x_i;i \in \bb N\}$ be the set of discontinuities of $W$. Defining $\bb R_W = \bb R \cup \{x_i-\}_i$ and taking $d_W(x_i-,y) = |W(x_i-)-W(y)|$, we see that $(\bb R_W, d_W)$ is a complete metric space. Now we notice that the functions $\varphi_n$ are $1/2$-H\"older continuous (in the usual sense) in $\bb R_W$. In particular, we can continuously extend $\varphi_n$ to the whole space $\bb R_W$. With respect to the usual topology of $\bb R$, this just means that $\varphi_n$ is c\`adl\`ag, and that the set of discontinuities of $\varphi_n$ is contained in $\{x_i\}_i$. In $\bb R_W$, the sequence $\{\varphi_n\}_n$ is uniformly equicontinuous, and by Arzel\`a-Ascoli theorem, there are a subsequence $n'$ and a function $\varphi$ such that $\varphi_{n'} \to \varphi$, uniformly in compacts. The function $\varphi$ inherits the H\"older-continuity from $\varphi_{n'}$. Considering c\`adl\`ag versions of the functions $\varphi_n$, $\varphi$, we see that $\varphi_{n'} \to \varphi$, uniformly in compacts, also with respect to the {\em usual topology} of $\bb R$. In particular, $\varphi_{n'} \to \varphi$ pointwise. By Fatou's lemma, $\int |W(x)| \varphi(x)^2 dx \leq 1$. Since $\varphi$ is c\`adl\`ag, $\int_{-1}^1 \varphi^2 dx< +\infty$ and we conclude that $\varphi \in \mc L^2(\bb R)$. Since $\varphi_{n'} \to \varphi$ uniformly in compacts, in order to have convergence in $\mc L^2(\bb R)$ as well we only need to have a uniform estimate on the tails of the integrals $\int (\varphi_{n'} -\varphi)^2$. For any $M >0$,
 \begin{align*}
  \limsup_{n' \to \infty} \int (\varphi_{n'} -\varphi)^2 dx
  	& \leq \limsup_{n' \to \infty} \int_{-M}^M (\varphi_{n'} - \varphi)^2 dx \\
	& \quad + 2 \limsup_{n' \to \infty} \int\limits_{[-M,M]^c} (\varphi_{n'}^2+\varphi^2) dx\\
	& \leq \frac{2}{\min\{-W(-M),W(M)\}} \int |W(x)|(\varphi_{n'}^2+\varphi^2) dx\\
	& \leq \frac{4}{\min\{-W(-M),W(M)\}}.
 \end{align*}
 
 Since $M$ is arbitrary, we conclude that $||\varphi_{n'} -\varphi||_0 \to 0$ as $n' \to \infty$. 
\end{proof}

As a simple consequence of this lemma, we see that the operator $\mc A$ has a compact resolvent. The operator $\mc A$ is also symmetric and non-negative. In particular, there is an orthonormal basis $\{\varphi_i\}_i$ of $\mc L^2(\bb R)$ formed by eigenvectors of $\mc A$. Define $\mc F_0= \text{span} \{\varphi_i\}_i$, the vectorial space finitely generated by $\{\varphi_i\}$. It is clear that $\mc F_0 \subseteq D(\mc A^n)$ for any $n>0$. Since both operators $\mc L_W$ and $|W(x)|$ are local, we can take the projection of $\varphi_i$ into a finite box to conclude that $\mc L_W \varphi_i$ is well defined, but at this point we can not ensure that $\mc L_W \varphi_i$ is an element of $\mc L^2(\bb R)$. 

\begin{lemma}
\label{l6}
 For any function $\varphi \in \mc H_2$, $\mc L_W \varphi \in \mc L^2(\bb R)$. 
\end{lemma}
 
\begin{proof}
 Let us compute $\<\mc A^2 \varphi, \varphi\>$ for $\varphi \in D(\mc A)$:
 \[
  \<\mc A^2 \varphi,\varphi\> = \<\mc A \varphi, \mc A \varphi\> = || \mc L_W \varphi ||^2 + || |W|^{1/2} \varphi ||^2 + 2\<-\mc L_W \varphi, |W| \varphi\>.
 \]
 
 Let us compute the crossed term $\<-\mc L_W \varphi, |W| \varphi\>$ with care.  By Remark \ref{r1}, there is a kernel for $\mc A^2$ composed only of functions with compact support. Take a function $\varphi \in D(\mc A^2)$ of compact support. In that case,
\begin{align*}
 \<- \mc L_W \varphi, |W| \varphi\> 
 	&= - \int |W| \varphi \frac{d}{dx} \frac{d\varphi}{dW} dx \\
	&= \int \frac{d \varphi}{dW} d\big\{|W|\varphi\big\}\\
	&= \int \frac{d \varphi}{d W} |W| d\varphi + \int \frac{d \varphi}{dW} \varphi d |W| \\
	&= \int |W| \Big\{\frac{d\varphi}{dW} \Big\}^2 dW + \int \frac{1}{2} \frac{d \varphi^2}{dW} d|W|\\
	&= \int |W| \Big\{\frac{d\varphi}{dW} \Big\}^2 dW - \frac{1}{2} \big(\varphi(0-)^2+ \varphi(0)^2\big),
\end{align*}
where we have performed an integration by parts in the second line and we have used Leibniz's rule $d(fg) = fdg+gdf$ in the third line. Therefore,
\[
 \<\mc A^2 \varphi, \varphi\> =||\mc L_W \varphi||_0^2 +|||W|^{1/2} \varphi||_0^2 + 2|||W|^{1/2} d\varphi/dW||_0^2 -\big(\varphi(0-)^2+\varphi(0)^2\big). 
\]

In particular, we have proved that 
\begin{equation}
 \label{ec6}
 \begin{split}
 ||\mc L_W \varphi||_0^2 
 	&\leq \<\mc A^2\varphi,\varphi\> + \varphi(0-)^2+\varphi(0)^2\\
 	&\leq ||\varphi||_2^2 +\varphi(0-)^2+\varphi(0)^2.
 \end{split}
\end{equation}

Therefore, our task is now to bound $\varphi(0)$ and $\varphi(0-)$ in terms of $||\varphi||_2$. Since $W$ is continuous at $x=0$, by Lemma \ref{l4} $\varphi(0-)=\varphi(0)$. If $\varphi(0)=0$, there is nothing to prove. Assume without loss of generality that $\varphi(0)>0$. Take a strictly increasing, continuous function $\Gamma: [0,\infty) \to [0,\infty)$ such that $\Gamma(0)=0$ and $|W(x)| \leq \Gamma(|x|)^2$ for any $x \in \bb R$. Since $|\varphi(x)-\varphi(0)| \leq ||\varphi||_1|W(x)|^{1/2}$, taking $\delta = \Gamma^{-1}(\varphi(0)/2||\varphi||_1)$ we have that $\varphi(x) \geq \varphi(0)/2$ for $x \in [-\delta,\delta]$. Therefore,
\[
 ||\varphi||_0 \geq \Big( \int_0^\delta \varphi(x)^2 dx\Big)^{1/2} \geq \frac{\varphi(0)}{2} \sqrt{\Gamma^{-1}\Big( \frac{\varphi(0)}{2||\varphi||_1}\Big)}.
\]

Define $\Theta(x) = x/2 \sqrt{\Gamma^{-1}(x/2)}$. Notice that $\Theta: [0,\infty) \to [0,\infty)$ is continuous and strictly increasing. Therefore, $\varphi(0) \leq ||\varphi||_1 \Theta^{-1}(||\varphi||_0/||\varphi||_1)$. Noticing that $||\varphi||_0 \leq ||\varphi||_1$, we have proved that $\varphi(0) \leq ||\varphi||_0 \Theta^{-1}(1)$. Putting this estimate into (\ref{ec6}), we conclude that there exists a constant $c=c(W)$ such that
\[
 || \mc L_W \varphi||_0 \leq c(W) ||\varphi||_2
\]
for any function of compact support. Taking suitable approximations, the lemma is proved for any function $\varphi \in \mc H_2$.
\end{proof}

Notice that the previous lemma is saying something more about $\mc L_W$. In fact, we can conclude that $\mc L_W$ is a continuous operator from $\mc H_2$ to $\mc L^2(\bb R)$. Define in $\mc F_0$ the metric
\[
 d(f,g) = \sum_{n \geq 0} 2^{-n} ||f-g||_n \wedge 1
\]
and let $\mc F$ be the closure of $\mc F_0$ under this metric. By definition, $\mc F \subseteq \mc H_n$ for any $n$ and by Lemma \ref{l5} the space $\mc F$ is a nuclear Fr\'echet space.

\begin{lemma}
 The operator $\mc L_W$ maps continuously $\mc F$ into itself.
\end{lemma}

\begin{proof}
 Repeating the arguments of Lemma \ref{l6}, we see that $\mc L_W$ maps continuously $\mc H_{n+2}$ into $\mc H_n$ for any $n \geq 0$.
\end{proof}

Now we are finally in position to prove the main result of this section.

\begin{theorem}
 \label{t1}
 Let $\mc Z_t$ be a semimartingale in $\mc F'$, the topological dual of $\mc F$ with respect to $\mc L^2(\bb R)$, and let $\mc Y_0$ be a distribution in $\mc F'$. There exists a unique solution $\mc Y_t$ of the equation
 \[
  d \mc Y_t = \mc L_W \mc Y_t dt + d \mc Z_t
 \]
with initial condition $\mc Y_0$.
\end{theorem}

\begin{proof}
 We just need to check if the conditions of Proposition \ref{p1} are fulfilled by $\mc L_W$ in $\mc F$. Since the operator $\mc L_W$ is bounded in $\mc F$, the semigroup $S_t = e^{t \mc L_W}$ is well defined by taking the corresponding power series, for example. For a given $\varphi \in \mc F$, the trajectories $S_t \varphi$ are continuous in $\mc F$ by the semigroup property and the boundedness of $\mc L_W$, so Proposition \ref{p1} applies in this situation.
\end{proof}

\section{The zero-range process}
\label{s3}

Let $p = \{p(x,y); x,y \in \bb Z^d\}$ be the jump rate of a continuous-time random walk on $\bb Z^d$. Let $g: \bb N_0 \to [0,\infty)$ be such that $g(0)=0$ and $g(n) \leq n$ for any $n \in \bb N_0$. The zero-range process with jump rate $p$ and interaction rate $g(\cdot)$ is a system of random walks on $\bb Z^d$ with the following dynamics. For each ordered pair of sites $(x,y) \in \bb Z^d \times \bb Z^d$, after an exponential time of rate $g(n) p(x,y)$, where $n$ is the number of particles on site $x$, one particle jumps from site $x$ to site $y$. Notice that the condition $g(0)=0$ ensures that there is at least one particle at site $x$ when the jump occurs. After a jump, a new exponential time starts afresh. This happens independently for each pair of sites $(x,y)$. 

Since $p$ is the jump rate of a continuous-time random walk on $\bb Z^d$, we have $\sum_{y \in \bb Z^d} p(x,y) <+\infty$ for any $x \in \bb Z^d$. Without loss of generality, we assume that $p(x,x)=0$ for any $x \in \bb Z^d$. Since we are interested in {\em symmetric} systems, we also assume that $p(x,y) = p(y,x)$ for any $x,y \in \bb Z^d$. Notice that the dynamics described so far is {\em conservative}, that is, particles are not created neither annihilated. When the initial number of particles is finite, this dynamics corresponds to a continuous-time Markov chain. Let $\Omega = \bb N_0^{\bb Z^d}$ be the space of configurations of particles for this process. We denote the state of the process at time $t$ by $\eta_t$, and we denote by $\eta_t(x)$ the number of particles at site $x$ at time $t$. When the initial number of particles is infinite, some conditions are needed in order to prevent explosions on the system, that is, the arrival of an infinite number of particles to some site in finite time. For our purposes, the Lipschitz condition $\sup_n|g(n+1)-g(n)|<+\infty$ on the interaction rates and $\sup_x E[\eta(x)] <+\infty$ on the initial (maybe random) distribution of particles will be enough (see \cite{And}, \cite{Lig} for a more detailed discussion). 

The process $\eta_t$ described above is a Markov process generated by the operator $L$ given by
\[
 L f(\eta) = \sum_{x,y \in \bb Z^d} g\big(\eta(x)\big) \big[ f(\eta^{x,y})-f(\eta)\big],
\]
where $f: \Omega \to \bb R$ is a local function (that is, it depends on $\eta(x)$ for a finite number of sites $x \in \bb Z^d$), $\eta^{x,y}$ is given by
\[
 \eta^{x,y}(z) =
 \begin{cases}
  \eta(x)-1, & z=x\\
  \eta(y)+1, & z=y\\
  \eta(z), &z \neq x,y
 \end{cases}
\]
and $f$ satisfies a Lipschitz condition:
\[
 \sup_{x, y, \eta} \big|f(\eta^{x,y})-f(\eta)\big| <+\infty.
\]

\subsection{Invariant measures}

Under the symmetry condition  $p(x,y)=p(y,x)$ for any $x,y$, the process $\eta_t$ has a family of invariant measures of product form. For $\varphi \geq 0$, let us define
\[
 Z(\varphi) = \sum_{n \geq 0} \frac{\varphi^n}{g(n)!},
\]
where $g(n)!= g(1)\cdots g(n)$ for $n \geq 1$ and $g(0)!=1$. Define
\[
 \varphi_c = \liminf_{n \to \infty} \sqrt[n]{g(n)!}.
\]

The power series defining $Z(\varphi)$ converges for $\varphi<\varphi_c$ and diverges for $\varphi>\varphi_c$. We assume that $\varphi_c>0$. This is the case if, for example, $\liminf_n g(n)>0$. For $\varphi<\varphi_c$, let $\bar \nu_\varphi$ be the product measure on $\Omega$ defined by
\[
 \bar \nu_\varphi\big( \eta(x_1)=n_1,\dots,\eta(x_l) = n_l\big) = \prod_{i=1}^l \frac{1}{Z(\varphi)} \frac{\varphi^{n_i}}{g(n_i)!}.
\]

The measure $\bar \nu_\varphi$ satisfies the detailed balance condition, and therefore it is left invariant by the evolution of $\eta_t$ and it is also reversible under $\eta_t$. Let us define the density of particles $\rho(\varphi) = \int \eta(x) \bar \nu_\varphi$. A simple computation shows that $\rho(\varphi) = \varphi Z'(\varphi)/Z(\varphi)$. Since $\varphi \mapsto \rho(\varphi)$ is strictly increasing and smooth, $\rho_c = \lim_{\varphi \to \varphi_c} \rho(\varphi)$ exists and $\rho:[0,\varphi_c) \to [0,\rho_c)$ is a diffeomorphism.

Since the number of particles is conserved by the dynamics, it is more natural to parametrize the family of invariant measures $\{\bar \nu_\varphi; \varphi \in [0,\varphi_c)\}$ by the density of particles. Denote by $\varphi(\rho)$ the inverse of $\rho(\varphi)$. We define $\nu_\rho = \bar \nu_{\varphi(\rho)}$.

The family $\{\nu_\rho; \rho \in [0,\rho_c)\}$ of invariant measures is in addition ergodic if the transition rate $p$ is irreducible and $g(n)>0$ for $n>0$. From now on, we assume that in fact the measures $\{\nu_\rho\}$ are ergodic under the evolution of $\eta_t$.

\subsection{Density fluctuations}

Let us fix $\rho \in (0,\rho_c)$ and let us take $\nu_\rho$ as the initial distribution for $\eta_t$. Let $x(t)$ be the continuous-time random walk associated to $p(\cdot)$. Assume that $x(t)$ satisfies an invariance principle in the following sense: there are a scaling sequence $\{a(n); n \in \bb N\}$ and a Markov process $\mc X(t)$ such that
\[
 \lim_{n \to \infty} \frac{1}{n} x\big(ta(n)\big) = \mc X(t),
\]
in distribution with respect to the Skorohod topology on the space of c\`adl\`ag paths $\mc D([0,\infty),\bb R^d)$. Let us denote by $\mc L$ the generator of the process $\mc X(t)$. Assume that there are a nuclear Fr\'echet space $\mc F$ and a Banach space $\mc V$ for which the conditions in Proposition \ref{p2} are satisfied. Let $\mc C_c(\bb R^d)$ be the space of continuous functions in $\bb R^d$ of compact support. Since the measure $\nu_\rho$ is of product form, for any function $G \in \mc C_c(\bb R^d)$,  
\[
 \lim_{n \to \infty} \frac{1}{n^d} \sum_{x \in \bb Z^d} \eta_t(x) G(x/n) = \rho \int G(x) dx,
\]
in probability with respect to the distribution of $\eta_t$. It is natural, therefore, to define the {\em density fluctuation field} $\mc Y_t^n$ as the process in $\mc D([0,\infty), \mc F')$ given by 
\[
 \mc Y_t^n(G) = \frac{1}{n^{d/2}} \sum_{x \in \bb Z^d} \big(\eta_{t a(n)} (x) -\rho \big) G(x/n)
\]
for any $G \in \mc F$. Since the measure $\nu _\rho$ is of product form and $\eta(x)$ has a bounded second moment with respect to $\nu_\rho$, for each fixed time $t \geq 0$, the field $\mc Y_t^n$ converges to a Gaussian field of mean 0 and covariance matrix $\chi(\rho) \delta(y-x)$, where $\chi(\rho) = \text{Var}(\eta(x);\nu_\rho)$ and $\delta(x)$ is the Dirac mass at $x$. In other words, $\mc Y_t^n$ converges to a spatial white noise. The fluctuation-dissipation principle states (at least in the diffusive setting) that there exists a constant $D(\rho)$, called the {\em diffusivity}, such that $\mc Y_t^n$ converges in distribution to the solution $\mc Y_t$ of the Ornstein-Uhlenbeck equation 
\[
 d\mc Y_t = D(\rho) \mc L^* \mc Y_t dt + d \mc Z_t
\]
with initial condition $\mc Y_0$ equal to the spatial white noise described above, where $\mc Z_t$is an $\mc F'$-valued martingale of quadratic variation $\<\mc Z_t\>$ characterized by
\[
 \<\mc Z_t\>(G) = \<\mc Z_t(G)\> = 2 D(\rho) \chi(\rho) t \<G, -\mc L G\>.
\]

In the context of the zero-range process, $D(\rho) = \varphi'(\rho)$ and $\chi(\rho) = \varphi(\rho)/\varphi'(\rho)$. Our aim is to give a rigorous proof of this principle in two situations: for a superdiffusive system on which $a(n)$ grows faster than $n^2$ and for a subdiffusive system on which $a(n)$ grows slower than the diffusive scaling $n^2$.

\subsection{The superdiffusive case}

In order to simplify the exposition and to concentrate in the technical problems posed by anomalous diffusive behavior of the associated random walks, from now on we assume that $g(n) =1$ for any $n>0$. Our first choice for the jump rates is $p(x,y) = q(y-x)$, where $q: \bb R^d \setminus \{0\} \to (0,\infty)$ is a continuous function satisfying $q(\lambda x) = \lambda^{-(d+\alpha)} q(x)$ for any $x \in \bb R^d \setminus \{0\}$, any $\lambda \neq 0$ and some $\alpha \in (0,2)$. The condition $\alpha>0$ is assumed in order to have $\sum_y p(x,y) <+\infty$. In the other hand, the condition $\alpha <2$ is assumed in order to have a non-Brownian scaling limit for the associated random walk. This choice of jump rates corresponds to the {\em zero-range process with long jumps} introduced in \cite{Jar}. It is well known that the associated random walk $x(t)$ satisfies an invariance principle:
\[
 \lim_{n \to \infty} x(tn^\alpha) = \mc X(t),
\]
where $\mc X(t)$ is a L\'evy process generated by the operator $\mc L$ whose action over smooth functions $F$ of compact support is given by
\[
 \mc L F(x) = \int\limits_{\bb R^d} q(y) \big(F(x+y)+F(x-y)-2F(x)\big) dy.
\]

Since $q(\cdot)$ is continuous and positive in the unit sphere, we see that $\mc L$ is bounded with respect to the fractional Laplacian $-(-\Delta)^{\alpha/2}$. In particular, $\mc L$ also satisfies the conditions of Proposition \ref{p2}, and therefore there exists a unique generalized solution of the Ornstein-Uhlenbeck equation
\begin{equation}
\label{ec7}
 d\mc Y_t  \varphi'(\rho) \mc L^* \mc Y_t dt + d \mc Z_t
\end{equation}
with initial distribution $\mc Y_0$ equal to  a mean-zero spatial white noise of covariance matrix $\chi(\rho) \delta(y-x)$ and driving martingale $\mc Z_t$ satisfying
\[
 \<\mc Z_t(G)\> = 2t \varphi(\rho) \iint dx dy q(y)\big(G(y)-G(x)\big)^2.
\]

A simple integration by parts shows that this last integral is equal to the energy form $\mc E(G,G)=: - \int G(x) \mc L G(x) dx$. Denote by $\eta_t^n$ the zero-range process $\eta_{tn^\alpha}$ associated to this choice of jump rates and with initial distribution $\eta_0^n$ given by $\nu_\rho$. We adopt the notation $\bb P^n$ for the distribution of the process $\eta_t^n$ in $\mc D([0,T], \Omega)$ and $\bb E^n$ for the expectation with respect to $\bb P^n$. 

\begin{theorem}
\label{t2}
 The fluctuation field $\mc Y_t^n$ defined by
 \[
  \mc Y_t^n(G) = \frac{1}{n^{d/2}} \sum_{x \in \bb Z^d} \big(\eta_t^n(x) \rho\big) G(x/n)
 \]
 converges to the Ornstein-Uhlenbeck process solving (\ref{ec7}), in distribution with respect to the Skorohod topology in $\mc D([0,T], \mc F')$.
\end{theorem}

\subsection{The subdiffusive case}
\label{s3.4}
In this section we define a process evolving in the one-dimensional lattice. A positive random variable $\zeta$ is said to be {\em $\alpha$-stable} ($\alpha \in (0,1)$), if it satisfies
\[
 \log E[e^{-\lambda \zeta}] = -c|\lambda|^\alpha
\]
for some constant $c>0$. Let $\{\xi_x; x \in \bb Z\}$ be a sequence of i.i.d. random variables with common distribution $\zeta$, where $\zeta$ is an $\alpha$-stable law. 

For a given realization of $\{\xi_x\}_x$, we define $p(x,x+1)=p(x+1,x)= \xi_x^{-1}$, and $p(x,y)=0$ if $|x-y| \neq 1$. It is well known that the sequence $\{\xi_x\}_x$ satisfies an invariance principle:
\[
 \lim_{n \to \infty} \frac{1}{n^{1/\alpha}} \sum_{i=1}^{[nx]} \xi_i = W(x)
\]
in distribution with respect to the Skorohod topology of $\mc D(\bb R, \bb R)$, where $W(x)$ is a double-sided, $\alpha$-stable subordinator with $W(0)=0$. Here we denote by $[nx]$ the integer part of $nx$. The following properties of the process $W(t)$ hold $a.s.$ The function $x \mapsto W(x)$ is c\`adl\`ag, increasing and of pure-jump type. We have $\lim_{n \to \pm \infty} W(x) = \pm \infty$. The function $W(x)$ is continuous at $0$, that is, $0$ is not a point of jump of $W(x)$.  
 
Consider the random walk $x(t)$ associated to the jump rate $p(\cdot,\cdot)$ defined above. The process $x(t)$ satisfies an invariance principle \cite{KK}, \cite{FJL}:

\begin{proposition}
 There exists a self-similar process $\mc X(t)$ of continuous paths, with self-similarity index $\alpha/(1+\alpha)$ such that 
 \[
  \lim_{n \to \infty} \frac{1}{n} x(t n^{1+1/\alpha}) = \mc X(t)
 \]
in distribution with respect to the Skorohod topology of $\mc D([0,\infty), \bb R)$. 
\end{proposition}

We point out that this is an {\em annealed} result: the convergence result is true only when averaged over the environment $\xi$. We will not give an explicit construction of the process $\mc X(t)$ here; for a detailed description of $\mc X(t)$, see \cite{KK} or \cite{FJL}. We just point out that $\mc X(t)$ is {\em not} Markovian, due to the averaging over the environment. Notice the subdiffusive scaling $1+1/\alpha>2$.

This proposition has been obtained by considering a different environment, distributed as $\xi$, but changing with the scaling. A simple computation shows that the distribution of $W(1)$ is equal to $\zeta$. A scaling argument shows that, in fact, any increment of the form $\epsilon^{-1/\alpha} \big(W(x+\epsilon)-W(x)\big)$ is distributed according to $\zeta$. Moreover, if the intervals $(x_0,x_1]$, $(y_0,y_1]$ are disjoints, the variables $W(x_1)-W(x_0)$, $W(y_1)-W(y_0)$ are independent. In view of these observations, for each $n$ we define 
\[
 \xi_x^n = n^{1/\alpha} \big\{ W\big((x+1)/n\big) - W\big(x/n\big)\big\}.
\]

Observe that for any fixed $n$, the sequence $\xi^n = \{\xi_x^n; x \in \bb Z\}$ has the same distribution of $\xi$, that is, it is a sequence of i.i.d. random variables with common distribution $\zeta$. In the other hand, by construction the normalized sums of the triangular array $\{\xi_x^n; x \in \bb Z, n \geq 1\}$ converge $a.s.$ to the process $W(\cdot)$. As above, we define the jump rate $p_n(\cdot,\cdot)$ by taking $p_n(x,x+1) = p_n(x+1,x) = (\xi_x^n)^{-1}$, and $p_n(x,y)=0$ if $|x-y| \neq 1$. Denote by $x_n(t)$ the random walk associated to $p_n(\cdot,\cdot)$. The sequence of random walks $\{x_n(t)\}_n$ satisfies the following invariance principle:

\begin{proposition}
 Fix a realization of $W(\cdot)$. Then,
 \[
  \lim_{n \to \infty} \frac{1}{n} x_n(tn^{1+1/\alpha}) = \mc X_W(t)
 \]
in distribution with respect to the Skorohod topology of $\mc D([0,\infty),\bb R)$, where $\mc X_W(t)$ is the Markov process generated by the operator $\mc L_W =  d/dx d/dW$. 
\end{proposition}

The main difference between this proposition and the previous one is that now the statement of this proposition is a {\em quenched} statement, since the random environment is first fixed, and then an invariance principle is claimed for that particular choice of the environment. We point out that the process $\mc X_W(t)$ is Markovian, but not strongly Markovian. For this and other properties of the process $\mc X_W(t)$, see \cite{FJL}.

Notice that almost any realization of $W(\cdot)$ satisfies the assumptions made in Section \ref{s2}. Therefore, Theorem \ref{t1} holds for the generator $\mc L_{W}$ and there is a unique solution of the Ornstein-Uhlenbeck equation 
\begin{equation}
 \label{ec4.5}
d \mc Y_{t}^W = \varphi'(\rho) \mc L_W \mc Y_{t}^W dt + \varphi(\rho)^{1/2} d \mc Z_{t}^W  
\end{equation}
with initial distribution $\mc Y_{0}$ equal to a mean-zero spatial white noise of covariance matrix $\chi(\rho) \delta(y-x)$ and driven by the martingale satisfying
\[
\< \mc Z_t^W \> = 2t \int \Big( \frac{dG}{dW}\Big)^2 dW.
\]

Let $\eta_t^{n,W}$ be the rescaled zero-range process $\eta_{tn^{1+1/\alpha}}$ associated to the transition rate $p_n(\cdot,\cdot)$. We call this process the {\em zero-range process with environment $W$}.

\begin{theorem}
\label{t3}
Fix a realization of $W$ and a density $\rho \in (0,\rho_c)$. Consider the zero-range process $\eta_t^{n,W}$ with initial distribution $\nu_\rho$. The fluctuation field $\mc Y_\cdot^{n,W}$ defined by
\[
\mc Y_t^{n,W}(G) = \frac{1}{\sqrt n} \sum_{x \in \bb Z} \big( \eta_t^{n,W}(x)-\rho\big)
\]
converges to the Ornstein-Uhlenbeck process $\mc Y_t^W$ defined above, in distribution with respect to the Skorohod topology in $\mc D([0,\infty),\mc F')$. 
\end{theorem}

Notice that an {\em  annealed} result for the process $\eta_t^{n,W}$ implies an analogous result for the process $\eta_t^n$. In order to state a fluctuation result for the empirical density associated to $\eta_t^n$, we need first some definitions. The main difficulty is that the space where $\mc Y_\cdot^W$ is defined depends on $W$. Remember that for any $W$, any $t \geq$ and any $G \in \mc C_c(\bb R)$ the random variable $\mc Y_t^W(G)$ is well defined. In fact, the finite-dimensional distributions $(\mc Y_{t_1}^W(G),\dots,\mc Y_{t_l}^W(G))$ are also well defined. What is not well defined, is $\mc Y_t^W(G)$ as a process with c\`adl\`ag paths. By Kolmogorov's consistency theorem, we can consider $\mc Y_\cdot^W$ as a process defined in $\mc M(\bb R)^{[0,\infty)}$, where $\mc M(\bb R)$ is the space of Radon measures in $\bb R$. This space is very irregular, but it has the advantage that it does not depend on $W$. Let $\bb P^W$ be the distribution of $\mc Y_\cdot^W$ in $\mc M(\bb R)^{[0,\infty)}$. We denote by $\mc Y^{\text{ann}}_\cdot$ the process defined in $\mc M(\bb R)^{[0,\infty)}$ and having law $P \otimes \bb P^W$.
The following result is an immediate consequence of Theorem \ref{t3}:

\begin{corollary}
\label{c1}
Fix $\rho \in (0,\rho_c)$. Let $\eta_t^\xi$ be the zero-range process with jump rate $p^\xi(\cdot,\cdot)$ and initial distribution $\nu_\rho$. The fluctuation field $\mc Y_\cdot^n$ defined by
\[
\mc Y_t^n(G) = \frac{1}{\sqrt n} \sum_{x \in \bb Z} \big( \eta_{t n^{1+1/\alpha}}^\xi(x)-\rho\big)G(x/n)
\]
for $G \in \mc C_c(\bb R)$ converges in the sense of finite-dimensional distributions to $\mc Y_\cdot^{\text{ann}}$.
\end{corollary}

As the invariance principle for $x(t)$ discussed before, this statement is an {\em annealed} result, which holds only after averaging over the environment. As we discussed above, the nuclear space $\mc F$ constructed in Section \ref{s2} {\em depends on the realization} of $W$. In particular, we do not expect a functional central limit theorem to hold for $\mc Y_t^n$, since there is no good set of test functions $G$ that works for any choice of the environment. Notice that we have constructed the process $\mc Y_\cdot^{\text{ann}}$ only to obtain this corollary, which does not involve the test space $\mc F$. 

\subsection{Fluctuations of the current of particles}
\label{s4.5}
In this section we restrict ourselves to dimension $d=1$ and we consider either the zero-range process with long jumps for $1<\alpha<2$ or the zero-range process with environment $W$ for $0<\alpha<1$. Given two different sites $x,y \in \bb Z$, we define $J_{x,y}(t)$ as the total current of particles between $x$ and $y$ up to time $t$. That is, $J_{x,y}(t)$ counts the number of jumps from $x$ to $y$ minus the number of jumps from $y$ to $x$, up to time $t$. It is not difficult to see that $J_{x,y}(t)$ is a well defined Poisson process, regardless of the dimension or the parameters of the zero-range process. With our definitions, $J_{x,y}(t)= - J_{y,x}(t)$, and for the zero-range process with environment $W$, $J_{x,y}(t)=0$ unless $x$,$y$ are neighbors. The conservation of the number of particles can be stated as a continuity equation for $\eta_t(x)$: 
\[
\eta_t(x) - \eta_0(x) = - \sum_{y \in \bb Z} J_{x,y}(t).
\]

For $x \in \bb Z$, we define $J_x(t)$ as the total current of particles through the bond $\<x,x+1\>$, that is,
\[
J_x(t) = \sum_{\substack{y \leq x \\ z>x}} J_{y,z}(t).
\]

The inhomogeneous Poisson process $J_{x,y}(t)$ is a birth and death process with birth rate $p(x,y) g(\eta_t(x))$ and death rate $p(x,y)g(\eta_t(y))$,  and therefore the birth and death rates of the process $J_x(t)$ are respectively
\[
\sum_{\substack{y \leq x \\ z>x}} p(y,z) g(\eta_t(y)), \sum_{\substack{y \leq x \\ z>x}} p(z,y) g(\eta_t(z)).
\] 

For the zero-range process in random environment $W$, these two sums are equal to $p_n(x,x+1)g(\eta_t(x))$ and $p(x,x+1)g(\eta_t(x+1))$ respectively. For the zero-range process, these sums are convergent for ``reasonable'' configurations of particles (for example, $a.s.$ with respect to $\nu_\rho$) precisely for $\alpha>1$.  When the number of particles is finite, the continuity equation tells us that 
\begin{equation}
\label{ec5}
J_x(t) = \sum_{y>x} \big(\eta_t(y)-\eta_0(y)\big),
\end{equation}
that is, the current through the bond $\<x,x+1\>$ up to time $t$ is equal to the number of particles to the right of $x$ at time $t$, minus the number of particles to the right of $x$ at time 0. Although equation (\ref{ec5}) does not make sense for the process starting from the measure $\nu_\rho$, let us assume for a moment that (\ref{ec5}) holds in that case. We will have
\begin{align*}
\frac{1}{\sqrt n} J_0(tn^\gamma) 
	&=\frac{1}{\sqrt n} \sum_{y>0} \big(\eta_t(y)-\eta_0(y)\big) -\frac{1}{\sqrt n} \sum_{y>0} \big(\eta_0^n(y)-\rho\big)\\
	&= \mc Y_t^n(H_0) - \mc Y_0^n(H_0),
\end{align*}
where $H_0(x)= \mathbf{1}(x>0)$ and $\gamma=1/\alpha$ for the zero-range process with long jumps and $\gamma = 1+1/\alpha$ in the case of the zero-range process with environment $W$. Of course, $\mc Y_t^n(H_0)$ is not well defined, but at least heuristically we can guess that $n^{-1/2} J_0(tn^\gamma)$ converges as $n \to \infty$ to $\mc Y_t(H_0) -\mc Y_0(H_0)$.

The key idea introduced by Rost and Vares \cite{RV} is that the {\em difference} $\mc Y_t^n(H_0) - \mc Y_0^n(H_0)$ is well defined and it is well behaved with respect to the limit in $n$. The asymptotic limits for the current are the following:

\begin{theorem}
 \label{t4}
 \begin{itemize}
 \item[]
  \item[i)] For the zero-range process with long jumps starting from the initial distribution $\nu_\rho$, 
  \[
   \lim_{n \to \infty} \frac{1}{n} J_0(tn^{2\alpha}) = \mc X(t)
  \]
  in the sense of convergence of finite-dimensional distributions, where $\mc X(t)$ is fractional Brownian motion of Hurst exponent $1/2\alpha$. 
  
  \item[ii)] For the zero-range process with environment $W$ and starting from the initial distribution $\nu_\rho$, 
  \[
  \lim_{n \to \infty} \frac{1}{n} J_0(tn^{2+2/\alpha}) = \mc X^W(t)
  \]
  in the sense of convergence of finite-dimensional distributions, where $\mc X(t)$ is a Gaussian process with mean zero and covariance matrix depending on $W$.
  
  \item[iii)] Let us denote by $J_0^\xi(t)$ the current associated to the zero-range process with random environment $\xi$. Under the conditions of $ii)$,
  \[
  \lim_{n \to \infty} \frac{1}{n} J_0(tn^{2+2/\alpha}) = \mc X(t)
  \]
  with respect to the annealed distribution $P^\xi \otimes \bb P^n$ and in the sense of finite-dimensional distributions, where $\mc X(t)$ is a fractional Brownian motion of Hurst exponent $\alpha/(2+2\alpha)$.
 \end{itemize}
\end{theorem}

\section{Density fluctuations: the superdiffusive case}
\label{s4}
In this section we prove Theorem \ref{t2}. The proof follows the usual approach to convergence of stochastic processes, which consists of two steps. First we prove that the sequence of distributions of $\{\mc Y_\cdot^n\}_n$ in $\mc D([0,T],\mc F')$ is tight, from which we conclude that the sequence has at least one accumulation point $\mc Y_\cdot$, defined as a c\`adl\`ag process on $\mc F'$. After that, we prove that the process $\mc Y_\cdot$ is a solution of the Ornstein-Uhlenbeck equation (\ref{ec7}). By Proposition \ref{p2}, there is at most one of such solutions. We finally conclude that the sequence $\{\mc Y_\cdot^n\}_n$ is relatively compact and has only one limit point, and therefore it is convergent.

\subsection{Tightness of $\{\mc Y_\cdot^n\}_n$}
\label{s4.1}
In this section we prove tightness of $\{\mc Y_\cdot^n\}_n$. To begin, we recall Mitoma's criterion \cite{Mit}:

\begin{proposition}
Let $H$ be a separable Hilbert space, and let $\mc F \subseteq H$ be a nuclear Fr\'echet space. Let $\mc F'$ be the topological dual of $\mc F$ with respect to $H$, and let $\{\mc Q^n\}_n$ be a sequence of distributions in $\mc D([0,T],\mc F')$. For a given function $G \in \mc F$, let $\mc Q^{n,G}$ be the distribution in  $\mc D([0,T],\bb R)$ defined by $\mc Q^{n,G}(y(\cdot) \in A) =\mc Q^{n}(\mc Y(\cdot)(G) \in A)$. The sequence $\{\mc Q^n\}_n$ is tight if and only if  $\{\mc Q^{n,G}\}_n$ is tight for any $G \in \mc F$.  
\end{proposition}

Applied to the sequence $\{\mc Y_\cdot^n\}_n$, this theorem says that $\{\mc Y_\cdot^n\}_n$ is tight if and only if $\{\mc Y_\cdot^n(G)\}_n$ is tight for any $G \in \mc S(\bb R^d)$. By Dynkin's formula and after some manipulations, we see that
\begin{equation}
\label{ec8}
\mc Y_t^n(G) = \mc Y_0^n(G) + \int_0^t \frac{1}{n^{d/2}} \sum_{x,y \in \bb Z^d} \big\{ g(\eta_s^n(x))-\varphi(\rho)\big\} \mc L_n G(x/n) ds + \mc M_t^n(G),
\end{equation}
where $\mc M_t^n(G)$ is a martingale of quadratic variation
\[
\<\mc M_t^n(G)\> = 2 \int_0^t \frac{1}{n^{2d}} \sum_{x,y \in \bb Z^d} q\Big( \frac{y-x}{n} \Big) \Big(G(y/n)-G(x/n)\Big)^2 g\big(\eta_s^n(x)\big)ds.
\]

Here we use the notation $\mc L_n G$ for the discrete approximation of $\mc L G$ given by
\[
 \mc L_n G(x) = \frac{1}{n^d} \sum_{y \in \bb Z^d} q(y/n) \big(G(x+y/n)+G(x-y/n)-2G(x)\big).
\]

Notice that $\mc L_n G(x)$ is just a Riemann sum for $\mc L G(x)$, and due to the symmetry of the kernel $q(\cdot)$ and the smoothness of $G$, the sum converges to $\mc L G(x)$, uniformly in $x$. 

In order to prove tightness for the sequence $\{\mc Y_\cdot^n(G)\}_n$, it is enough to prove tightness for $\{\mc Y_0^n(G)\}_n$, $\{\mc M_\cdot^n(G)\}_n$ and the integral term in (\ref{ec8}). The easiest one is the initial condition: we have already observed that $\mc Y_0^n(G)$ converges to a normal random variable of mean zero and variance $\chi(\rho) \int G(x)^2 dx$. For the other two terms, we use {\em Aldous' criterion} which now we state.

\begin{proposition}[Aldous' criterion]
\label{p3}
A sequence of distributions $\{\mc P^n\}$ in the path space $\mc D([0,T],\bb R)$ is tight if:

\begin{itemize}
\item[i)] For any $t \in [0,T]$ the sequence $\{\mc P_t^n\}$ of distributions in $\bb R$ defined by $\mc P_t^n(A) = \mc P^n(y(t) \in A)$ is tight,

\item[ii)] For any $\epsilon >0$,
\[
\lim_{\delta>0} \limsup_{n \to \infty} \sup_{\substack{\tau \in \bb T_T\\ \theta \leq \delta}} \mc P^n\big[|y(\tau+\theta)-y(\tau)| >\epsilon\big] =0,
\]
\end{itemize}
where $\bb T_T$ is the set of stopping times bounded by $T$  and $y(\tau+\theta)=y(T)$ if $\tau+\theta>T$.
\end{proposition}

Now we prove tightness of the martingale term. By the optional sampling theorem, we have

\begin{align*}
\bb P^n \big[ \big|\mc M_{\tau+\theta}^n(G)- \mc M_\tau^n(G)\big|> \epsilon\big] 
	& \leq \frac{1}{\epsilon^2} \bb E^n \big[ \big\< \mc M_{\tau+\theta}^n(G)\big\> - \big\< \mc M_\tau^n(G)\big\>\big]\\
	& \leq \frac{2\theta}{\epsilon^2} \mc E_n(G,G), 
\end{align*}
where in the last inequality we have used the fact that $g(\cdot)$ is bounded, and we have denoted by $\mc E_n(G,G)$ the {\em Dirichlet form} associated to $\mc L_n$:
\[
 \mc E_n(G,G) = \frac{1}{n^{2d}} \sum_{x,y \in \bb Z^d} q\Big(\frac{y-x}{n}\Big) \Big(G(y/n)-G(x/n)\Big)^2. 
\]

In particular, the martingale $\mc M_t^n(G)$ satisfies the conditions of Aldous' criterion. The integral term can be treated in a similar way:

\begin{align*}
 \bb E^n\Big[ \Big( \int_{\tau}^{\tau+ \theta} \frac{1}{n^{d/2}} 
 	&\sum_{x \in \bb Z^d} \big(g(\eta_s^n(x))-\varphi(\rho)\big) \mc L_n G(x/n) ds \Big)^2 \Big] \leq \\
 	& \leq \theta \int_0^T \bb E^n \Big[\Big(\frac{1}{n^{d/2}} \sum_{x \in \bb Z^d} \big(g(\eta_s^n(x)) -\varphi(\rho)\big)\mc L_n G(x/n) \Big)^2 \Big] ds \\
	& \leq \frac{\theta T}{n^d} \sum_{x \in \bb Z^d} \text{Var} \big(g(\eta(x)), \nu_\rho\big) \big( \mc L_n G(x/n)\big)^2\\
	&\leq C \theta T
\end{align*}
for some constant $C$ which depends only on $G$ and $\rho$. Therefore, we conclude that the sequence $\{\mc Y_\cdot^n\}_n$ is tight. Notice that as a by-product of these computations, the sequence of $\mc F'$-valued martingales $\{\mc M_\cdot^n\}_n$ is also tight, as well as the sequence of integrals.

\subsection{The Boltzmann-Gibbs principle}
\label{s4.2}
Let $\mc Y_\cdot$ be a limit point of the sequence $\{\mc Y_\cdot^n\}_n$. We want to prove that $\mc Y_\cdot$ is a generalized solution of (\ref{ec7}). It is enough to prove the following result.

\begin{theorem}
\label{t5}
 For any function $G \in \mc F$,
 \begin{itemize}
  \item[i)] The integral
  \[
   \int_0^T \mc Y_t(\mc L G) ds
  \]
is well defined as a random variable in $\mc F'$.
  \item[ii)] The process
  \[
   \mc M_t(G) = \mc Y_t(G) - \mc Y_0(G) - \varphi'(\rho) \int_0^t \mc Y_s(\mc L G) ds
  \]
 is a martingale of quadratic variation $\< \mc M_t(G)\> =2t \varphi(\rho) \mc E(G,G)$. 
 \end{itemize}
\end{theorem}

We will assume this theorem and we will finish the proof of Theorem \ref{t2}. We have already seen that for any fixed time $t$, the $\mc F$-valued random variable $\mc Y_t$ is a Gaussian field of mean zero and covariance matrix $\chi(\rho) \delta(y-x)$. Remember that $\{\mc M_\cdot^n\}_n$ is tight. Taking an adequate subsequence, we can assume that there is an $\mc F'$-valued process $\mc M_t$ for which $\mc M_t(G)$ is as above. Since all the projections of $\mc M_t$ are martingales, we conclude that $\mc M_t$ is an $\mc F'$-valued martingale. Taking piecewise linear approximations, it is not difficult to see that for any $\psi_t \in \mc F_T$, 
\[
 \mc Y_t(\psi_t) - \mc Y_0(\psi_0) - \int_0^t \mc Y_s \Big( \frac{\partial}{\partial t} \psi_s + \mc L \psi_s\Big) ds
\]
is a martingale of quadratic variation $2\chi(\rho) \int_0^t \mc E(\psi_s,\psi_s) ds$. This martingale is simply equal to the stochastic integral
\[
 \int_0^t d \mc M_s(\psi_s) ds = \mc M_t(\psi_t) - \int_0^t \mc M_s\Big(\frac{\partial}{\partial t} \psi_s\Big) ds.
\]

For $t=T$ we have
\[
 \mc Y_0(\psi_0) + \int_0^T \mc Y_t \Big(\frac{\partial}{\partial t} \psi_t + \mc L \psi_t \Big) dt = \int_0^T \mc M_t\Big(\frac{\partial}{\partial t} \psi_t \Big) dt,
\]
which proves that $\mc Y_t$ is a generalized solution of (\ref{ec7}).

Before entering into the proof of Theorem \ref{t5}, we state a proposition which is known in the literature as the {\em Boltzmann-Gibbs principle}, which was introduced by Rost \cite{Ros}:

\begin{proposition}[Boltzmann-Gibbs principle]
\label{p4}
 Let $H: \bb R^d \to \bb R$ be a continuous function of compact support. Then for any $t \in [0,T]$, 
 \[
  \lim_{n \to \infty} \bb E^n \Big[ \Big( \int_0^t \frac{1}{n^{d/2}} \sum_{x \in \bb Z^d} \big(g(\eta_s^n(x)) - \varphi(\rho) - \varphi'(\rho)(\eta_s^n(x) -\rho)\big) H(x/n) ds \Big)^2\Big] =0.
 \]
\end{proposition}

The idea behind this proposition is the following. Fluctuations of non conserved quantities are faster than fluctuations of conserved quantities (the number of particles in our case). Therefore, when we look at the right scaling and when averaged over time, only the projection of the fluctuations over the conserved field (the density field in our case) are seen. For the case of the zero-range process with long jumps, the proof of this proposition follows the proof for the diffusive case with slight modifications, so we omit it and we refer to Chapter 11.1 of \cite{KL} for a proof. We just point out an ingredient of the proof that can be easily overlooked, which is the ellipticity condition $\inf \{q(z); |z|=1\}>0$. This condition allows to compare the dynamics with the one associated to the fractional Laplacian $\Delta_\alpha$. We will give a more detailed discussion about the Boltzmann-Gibbs principle in Section \ref{s5}.

\begin{proof}[Proof of Theorem \ref{t5}]
We start by proving $i)$. The main point is that $\mc LG$ is {\em not} in $\mc F$ and by tightness we only know that $\mc Y_t(H)$ is well defined for $H \in \mc F$. Our first step is to prove that the operator $G \mapsto \int_0^t \mc Y_s(\mc L G) ds$ is continuous as a linear functional in $\mc C_{p,0}$ (remember the condition $d/2 < p < (d+\alpha)/2$). The operator $\mc L$ is continuous from $\mc F$ to $\mc C_{p,0}$. But
  \begin{align*}
   \bb E^n\big[ \big(\mc Y_t^n(H)\big)^2\big] 
   	&= \frac{1}{n^d} \sum_{x \in \bb Z^d} H(x/n)^2 \chi(\rho) \leq \frac{\chi(\rho) ||H||_{p,\infty}}{n^d} \sum_{x \in \bb Z^d} (1+(x/n)^2)^{-p} \\
		&\leq C(\rho,p) ||H||_{p,\infty}.
  \end{align*}

  This computation, when applied to the integral $\int_0^t \mc Y_s^n(H) ds$ gives
  \[
   \bb E^n \Big[ \Big( \int_0^t \mc Y_s^n(H) ds \Big)^2\Big] \leq C(\rho,p) t^2 ||H||_{p,\infty}.
  \]

  Therefore, for a sequence $\{H_l\}_l$ of functions in $\mc F$ which is of Cauchy in $\mc C_{p,0}$, the sequence $\{\int_0^t \mc Y_s^n(H) ds\}$ is a Cauchy sequence in $\mc L^2(\bb P^n)$, {\em uniformly} in $n$. We conclude that the random functional $\int_0^t \mc Y_s^n(H) ds$ extends to a random, continuous functional in $\mc C_{p,0}$. Up to here, we have proved that the family of random variables $\{\int_0^t \mc Y_s(\mc L G) ds; G \in \mc F\}$ is linear and continuous. By Lemma \ref{l2.1}, there is a unique $\mc F'$-valued random functional associated to this family, which proves $i)$. 
  
  Now we continue proving $ii)$. Using the Boltzmann-Gibbs principle into (\ref{ec8}), we see that
  \[
   \mc Y_t^n(G) =\mc Y_0^n(G) +\varphi'(\rho) \int_0^t \mc Y_s^n( \mc L G) ds + \mc M_t^n(G)
  \]
plus a rest that vanishes in $\mc L^2(\bb P^n)$ as $n \to \infty$. Remember that our proof of tightness actually proved that all the terms above are tight. Taking further subsequences if necessary, we obtain that
\[
 \mc Y_t(G) = \mc Y_0(G) + \varphi'(\rho)\int_0^t \mc Y_s(\mc LG) ds + \mc M_t(G)
\]
for some $\mc F'$-valued process $\mc M_t$. Notice the convergence of the integral term does not follows directly from the convergence of $\mc Y_\cdot^n$, but from $i)$. The convergence in distribution of the $\mc F'$-valued processes $\{\mc M_\cdot^n\}_n$ does not allow to conclude that $\mc M_\cdot$ is an $\mc F'$-valued martingale. Notice the simple bound $\sup_n \bb E^n[\mc M_T^n(G)^2]<+\infty$ for any $G \in \mc F$. We claim that this bound is enough to conclude that $\mc M_\cdot$ is a martingale. In fact, from this bound we conclude that the sequence $\{\mc M_T^n(G)\}_n$ has a subsequence converging to $\mc M_t(G)$ with respect to the weak topology of $\mc L^2(\bb P^n)$. Therefore, for a measurable set $U$ with respect to the canonical $\sigma$-algebra $\mathfrak F_t$, $\bb E^n[\mc M_t^n(G)\mathbf 1_U]$ converges to $E[\mc M_t(G) \mathbf 1_U]$. Since $\mc M_\cdot^n(G)$ is a martingale, $\bb E^n[\mc M_T^n(G) \mathbf 1_U] = \bb E^n[\mc M_t^n(G) \mathbf 1_U]$. And taking a further subsequence if necessary, this last term converges to $E[\mc M_t(G) \mathbf 1_U]$, which proves that $\mc M_\cdot(G)$ is a martingale for any $G \in \mc F$.

In order to finish the proof, we only need to obtain the quadratic variation of $\mc M_t(G)$. A simple application of Tchebyshev's inequality proves that $\<\mc M_t^n(G)\>$ converges in probability to $2t \varphi(\rho) \mc E(G,G)$ as $n \to \infty$. Remember the definition of quadratic variation. We need to prove that $\mc M_t(G)^2 - 2\varphi(\rho)t\mc E(G,G)$ is a martingale. The same argument we used above applies now if we can show that $\sup_n \bb E^n[\mc M_T^n(G)^4]<+\infty$ and $\sup_n \bb E^n[\<\mc M_T^n(G)\>^2]<+\infty$. Both bounds follows easily from the explicit form of $\<\mc M_t^n(G)\>$ and (\ref{ec8}).
  \end{proof}

\section{Density fluctuations: the subdiffusive case}
\label{s5}
In this Section we prove Theorem \ref{t3}. The scheme of the proof is the same we followed for the superdiffusive case. That is, we first prove that the fluctuation field $\{\mc Y_\cdot^{n,W}\}_n$ is tight and then we prove that any limit point of $\{\mc Y_\cdot^{n,W}\}_n$ is a solution of the generalized Ornstein-Uhlenbeck equation (\ref{ec4.5}). With respect to the superdiffusive case, we point out two differences. First, the operator $\mc L_W$ is continuous in the {\em ad-hoc} nuclear space $\mc F$, and therefore the conditions required for uniqueness of solutions of (\ref{ec4.5}) are simpler. In the other hand, the environment is {\em degenerated}, in the sense that the trajectory $W$ has a dense set of jumps and therefore there are rates $\xi_x^n$ arbitrarily small at any scale and any macroscopic region. This fact makes the system less ergodic, and therefore the proof of the Boltzmann-Gibbs principle will be more involved.

From now on we take a fixed realization of the environment $W$, so we will omit in the notation the dependence in $W$ of different quantities we will define. In order to profit from the concepts already introduced in the superdiffusive case, we denote by the same symbols analogous objects, like the generator $\mc L = d/dx d/dW$, the rescaled process $\eta_t^n$, etc. The reader should use the corresponding subdiffusive definitions of the various quantities of interest, instead of the ones used in the previous section.

\subsection{Tightness of $\{\mc Y_\cdot^n\}_n$}

Before entering into the proofs, we need some definitions. For each $n>0$ and each function $G: \bb R \to \bb R$ we define
\[
 \mc L_n G(x/n) = n \Bigg\{ \frac{G\big((x+1)/n\big) - G\big(x/n\big)}{W\big((x+1)/n\big) - W\big(x/n\big)} - 
 	\frac{G\big((x/n\big) - G\big((x-1)/n\big)}{W\big(x/n\big) - W\big((x-1)/n\big)} \Bigg\}.
\]

The operator $\mc L_n$ is a discrete approximation of $\mc L = d/dx d/dW$ and corresponds to the generator of the process $n^{-1} x_n(t n^{1+1/\alpha})$, where $x_n(t)$ is the random walk associated to $p_n$. It will be useful as well to define the {\em energy form}
\[
 \mc E_n(G,G) = \sum_{x \in \bb Z} \Bigg\{ \frac{G\big((x+1)/n\big) - G\big(x/n\big)}{W\big((x+1)/n\big) - W\big(x/n\big)} \Bigg\}^2 \Big( W\big((x+1)/n\big) - W\big(x/n\big) \Big).
\]

In the definition of $\mc E_n(G,G)$ we have not cancelled out the differences $W((x+1)/n)-W(x/n)$ in order to stress that $\mc E_n(G,G)$ is nothing but a Stieltjes sum for $\int (dG/dW)^2 dW$. In particular, for $G \in \mc F$ the energy $\mc E_n(G,G)$ converges to $\mc E(G,G) = \int (dG/dW)^2 dW$ as $n \to \infty$.

The proof of tightness for $\{\mc Y_\cdot^n\}_n$ is very similar to the proof in Section \ref{s4.1}. By Mitoma's criterion, it is enough to prove tightness for $\{\mc Y_\cdot^n(G)\}_n$ for any $G \in \mc F$. By Dynkin's formula,
\begin{equation}
\label{ec9}
 \mc Y_t^n(G) = \mc Y_0^n(G) + \int_0^t \frac{1}{n^{1/2}} \sum_{x \in \bb Z} \big(g(\eta_s^n(x))-\varphi(\rho)\big) \mc L_n G(x/n) ds + \mc M_t^n(G),
\end{equation}
where $\mc M_t^n(G)$ is a martingale of quadratic variation
\[
 \<\mc M_t^n(G) \> = \int_0^t \sum_{x \in \bb Z} \big(g(\eta_s^n(x))+g(\eta_s^n(x+1))\big)
 	\frac{\big(G((x+1)/n) - G(x/n)\big)^2}{W((x+1)/n) - W(x/n)} ds.
\]

Tightness of the martingales $\{\mc M_\cdot^n(G)\}_n$ follows from the deterministic bound $\<\mc M_t^n(G) \> \leq 2t \mc E_n(G,G)$. 
Tightness of the integral term follows after approximating $\mc L_n G(x/n)$ by $\mc L G(x/n)$ and using that $\mc L G$ is regular and square integrable. 
Since $\mc Y_0^n$ converges to a white noise of covariance matrix $\chi(\rho) \delta(y-x)$, we conclude that $\{\mc Y_\cdot^n\}_n$ is tight. 

\subsection{An estimate of energy type}

In this section we will assume the Boltzmann-Gibbs principle. We will prove that the limit points of $\{\mc Y_\cdot^n\}_n$ are thus concentrated on solutions of (\ref{ec4.5}). A key {\em energy estimate} (see Theorem \ref{t6}) will be the main tool allowing to prove the characterization of limit points of $\{\mc Y_\cdot^n\}_n$, explaining the name of the section. 

In the statement of the Boltzmann-Gibbs principle (Proposition \ref{p4}), we ask the test function $H$ to be continuous and of compact support. This point is crucial in the proof of Proposition \ref{p4}, since it allows to perform a space average that in the end reduces the proof to a version of the ergodic theorem. First of all, it is not clear if there are functions of compact support in $\mc F$. Second, they are definitively not continuous, since they are differentiable with respect to $W$, although they are c\`adl\`ag, leaving some margin for development. Moreover, a simple computation shows that the convergence of $\mc L_n G$ to $\mc L G$ holds in a weak sense, enough to establish the $\mc L^2$ estimates needed to prove tightness, but not enough to ensure some sort of continuity for $\mc L_n G$. The key fact here is that for any test function $G \in \mc F$, the derivative $dG/dW$ is a smooth function, since we have to take its usual derivative to compute $\mc L G$. In order to benefit from this continuity, we rewrite the decomposition (\ref{ec9}) as
\[
 \mc Y_t^n(G) = \mc Y_0^n(G) + \int_0^t \frac{1}{n^{1/2}} \sum_{x \in \bb Z} n\big(g(\eta_s^n(x))-g(\eta_s^n(x+1))\big) \nabla_x^n G ds + \mc M_t^n(G),
\]
where we have defined
\[
 \nabla_x^n G = \frac{G\big((x+1)/n\big) -G\big(x/n\big)}{W\big((x+1)/n\big) -W\big(x/n\big)}.
\]

In other words, $\nabla_x^n G$ is a discrete approximation of $dG/dW$. At this point, it is not clear at all that rewriting the integral term in this form is of some help, because now we have a big factor $n$ in front of $g(\eta_s^n(x))-g(\eta_s^n(x+1))$, and this last difference is not small, though it should be small after time integration. Denote by $\<\cdot,\cdot\>_\rho$ the inner product in $\mc L^2(\nu_\rho)$. We start recalling Kipnis-Varadhan inequality, valid for any reversible Markov process.

\begin{proposition}
 \label{p5}
 For any function $h: \Omega \to \bb R$ in $\mc L^2(\bb \nu_\rho)$ with $\int h d
 \nu_\rho =0$, 
 \[
  \bb E^n \Big[ \Big(\int_0^t h(\eta_t^n) ds \Big)^2 \Big] \leq 20t ||h||_{-1,n}^2
 \]
where the Sobolev-type norm $||\cdot||_{-1,n}$ is defined by
\[
 ||h||_{-1,n}^2 = \sup_f \big\{2\<f,h\> - \<f,-L_n f\>_\rho\big\},
\]
$L_n$ is the generator of the process $\eta_t^n$ and the supremum is over functions $f$ in $\mc L^2(\nu_\rho)$. 
\end{proposition}
 
Of course, in this proposition nothing prevents us from having $||h||_{-1,n}=+\infty$, but estimating $||h||_{-1,n}$ is usually part of the work when applying this inequality. 

\begin{theorem}
\label{t6}
 For any function $H: \bb R \to \bb R$,
 \begin{multline*}
  \bb E^n \Big[ \Big( \int_0^t \frac{1}{n^{1/2}} \sum_{x \in \bb Z} n \big(g(\eta_s^n(x)) - g(\eta_s^n(x+1)) \big) H(x/n) ds \Big)^2 \Big] \\
  	\leq 5t \varphi(\rho)^{-1} \sum_{x \in \bb Z} H(x/n)^2 \big(W((x+1)/n) -W(x/n)\big).
 \end{multline*}

\end{theorem}

\begin{proof}
Let us call $h(\eta_s^n)$ the term inside the integral. By Proposition \ref{p5}, the expectation is bounded by $20 t ||h||_{-1,n}^2$. Looking at the variational formula for $ ||h||_{-1,n}^2$, it will be good to obtain a more explicit formula for $\<f, -L_n f\>_\rho$. After some changes of variables and some algebra, it is not difficult to see that
\[
\<f, -L_n f\>_\rho = \sum_{x \in \bb Z} \frac{n}{W\big(\frac{x+1}{n}\big) -W\big(\frac{x}{n}\big)} \int g\big(\eta(x)\big) \big[f(\eta^{x,x+1})-f(\eta)\big]^2 \nu_\rho(d\eta).
\]

Use now the weighted Cauchy-Schwartz inequality $ab \leq a^2 \beta/2 +b^2/2 \beta$ with $a = \sqrt{g(\eta(x))}$ and $b = \sqrt{g(\eta(x))}[f(\eta^{x,x+1})-f(\eta)]$ to get
\begin{align*}
\int \big( g(\eta(x+1))-g(\eta(x))\big) f(\eta)d \nu_\rho
	&\leq \frac{1}{2\beta_x} \int g(\eta(x)) \big[f(\eta^{x,x+1})-f(\eta)\big]^2\nu_\rho(d\eta) \\
	&+ \frac{\beta_x}{2} \int g(\eta(x))d \nu_\rho.
\end{align*}

Notice that the last integral is bounded by $\varphi(\rho) \beta_x/2$. The factor we are estimating appears multiplied by $n^{1/2} H(x/n)$ in $\<f^2,h\>_\rho$. Therefore, choosing $\beta_x = |H(x/n)|\big(W((x+1)/n)-W(x/n)\big)/2n^{1/2}$ we obtain the bound
\begin{align*}
||h||_{-1,n}^2
	&\leq \frac{\varphi(\rho)}{2} \sum_{x \in \bb Z} n^{1/2} |H(x/n)| \beta_x\\
	&\leq \frac{\varphi(\rho)}{4} \sum_{x \in \bb Z} H(x/n)^2 \big(W((x+1)/n)-W(x/n)\big),
\end{align*}
which proves the theorem. 
\end{proof}

Remember that our goal is to replace in equation (\ref{ec9}) the term $(g(\eta_s^n(x))-\rho) \mc L_n G(x/n)$ by $\varphi'(\rho)(\eta_s^n(x)-\rho) \mc L G(x/n)$. Let us take now $H = dG/dW$. For $x \in \bb R$ and $\epsilon >0$,
\begin{align*}
G(x+\epsilon) -G(x) 
	&= \int_x^{x+\epsilon} H(y) W(dy) \\
	&= H(x)\big(W(x+\epsilon)-W(x)\big) + \int_x^{x+\epsilon} \big(H(y)-H(x)\big) W(dy).
\end{align*}

In particular,
\begin{equation}
\label{ec10}
\begin{split}
    \bigg|\frac{G(x+\epsilon)-G(x)}{W(x+\epsilon)-W(x)}-H(x)\bigg|
    & \leq \frac{1}{W(x+\epsilon)-W(x)} \int_x^{x+\epsilon} \big|H(y)-H(x)\big|W(dy)  \\
    &  \leq \sup\{\epsilon H'(y); x\leq y\leq x+\epsilon\}.
\end{split}
\end{equation}

From this estimate, we can perform a first replacement in (\ref{ec9}).

\begin{lemma}
\label{l7}
\[
\lim_{n \to \infty} \bb E^n\Big[ \Big( \int_à^t \frac{1}{n^{1/2}} \sum_{x \in \bb Z} n\big(g(\eta_s^n(x)) -g(\eta_s^n(x+1))\big)\big(\nabla_x^nG-H(x/n)\big)ds\Big)^2\Big] =0.
\]
\end{lemma}

\begin{proof}
Take $M>0$ and restrict the sum above to $\{|x|\leq Mn\}$. Then the limit holds by Theorem \ref{t6} and the fact that $H'$ is right-continuous. In the set $\{|x|>Mn\}$, bound $|\nabla_x^nG-H(x/n)|$ by $|\nabla_x^nG|+|H(x/n)|$ and use Theorem \ref{t6} again to prove that the expectation goes to $0$ as $M \to \infty$, {\em uniformly} in $n$. 
\end{proof}

In particular, we have the decomposition
\[
\mc Y_t^n(G) = \mc Y_0^n(G) + \int_0^t \frac{1}{n^{1/2}} \sum_{x \in \bb Z} n\big(g(\eta_s^n(x)) -g(\eta_s^n(x+1))\big)H(x/n)ds + \mc M_t^n(G),
\]
plus a rest that vanishes in $\mc L^2(\bb P^n)$ as $n \to \infty$. We have done all this work only to justify the exchange of $\nabla_x^n G$ by $H$. Now we can perform the integration by parts to write the integral as
\[
\int_0^t \frac{1}{n^{1/2}} \sum_{x \in \bb Z} \big(g(\eta_s^n(x)-\varphi(\rho)\big)n\big(H((x+1)/n)-H(x/n)\big)ds.
\]

Now we have made appear a finite approximation of $H'(x/n) = \mc L G(x/n)$. Since $\mc L G$ is uniformly $W$-H\"older continuous, using the same splitting into a box of size $Mn$ used above, we can write the integral as
\[
\int_0^t \frac{1}{n^{1/2}} \sum_{x \in \bb Z} \big(g(\eta_s^n(x)-\varphi(\rho)\big)\mc L G(x/n)ds
\]
plus a rest  that vanishes in $\mc L^2(\bb P^n)$ as $n \to \infty$. Now, finally, we have put the c\`adl\`ag test function $\mc L G$ in place of $\mc L_n G$. For the sake of completeness, we repeat here the statement of the Boltzmann-Gibbs principle. We will give an outline of the proof in the next section.

\begin{theorem}
\label{t7}
For any continuous function of compact support $H: \bb R \to \bb R$ we have
 \[
  \lim_{n \to \infty} \bb E^n \Big[ \Big( \int_0^t \frac{1}{n^{d/2}} \sum_{x \in \bb Z^d} \big(g(\eta_s^n(x)) - \varphi(\rho) - \varphi'(\rho)(\eta_s^n(x) -\rho)\big) H(x/n) ds \Big)^2\Big] =0.
 \]
\end{theorem}

As a corollary, in view of Lemma \ref{l7} we have the decomposition
\[
\mc Y_t^n(G) = \mc Y_0^n(G) +\varphi'(\rho) \int_0^t \mc Y_s^n(\mc L G) ds + \mc M_t^n(G),
\]
plus a rest  that vanishes in $\mc L^2(\bb P^n)$ as $n \to \infty$.  Take now a subsequence $n'$ such that $\{\mc Y_\cdot^n\}_n$ converges in distribution to some process $\mc Y_\cdot$. We can assume, taking a further subsequence if needed, that $\{\mc M_\cdot^n\}$ also converges to some process $\mc M_\cdot$. Differently from the superdiffusive case, here the function $\mc L G$ belongs to $\mc F$ and therefore the limiting processes $\mc Y_\cdot$, $\mc M_\cdot$ satisfy
\[
\mc Y_t(G) = \mc Y_0(G) + \int_0^t \mc Y_s(\mc L G) ds + \mc M_t(G)
\]
for any function $G \in \mc F$. As we did in the superdiffusive case, now our task is to prove that $\mc M_t(G)$ is a martingale of quadratic variation $2t \varphi(\rho) \mc E(G,G)$. By the invariance of $\nu_\rho$ under the evolution of $\eta_t$, the integral $\bb E^n[\<\mc M_t(G)\>]$ converges to $2t \varphi(\rho) \mc E(G,G)$. Therefore, it is enough to show that the variance of $\<\mc M_t(G)\>$ goes to 0 as $n' \to \infty$. For the superdiffusive case, it was enough to take the variance of the integrand and to show that it goes to $0$. This is not true in the superdiffusive case, since the increments of the form $W(x+\epsilon)-W(x)$ do not go to 0 with $\epsilon$, at least not uniformly, due to the jumps of $W$. We state the needed result in the form of a lemma.

\begin{lemma}
\label{l8}
\[
\lim_{n \to \infty} \bb E^n\Big[\Big( \int_0^t \sum_{x \in \bb Z} \big(g(\eta_s^n(x)-\varphi(\rho)\big) \big(\nabla_x^n G\big)^2 \Delta_x^n W ds\Big)^2 \Big] =0,
\]
where we have used the notation $\Delta_x^n W = W((x+1)/n) - W(x/n)$. 
\end{lemma}

\begin{proof}
First notice that the sum under the integral in the previous expression is uniformly bounded by $2 \mc E_n(G,G)$. Therefore, we can restrict the sum to a finite box of arbitrary size $M$. By the energy estimate in Theorem \ref{t6}, we that for any $M>0$,
\begin{multline*}
\bb E^n\Big[\Big( \int_0^t \sum_{|x| \leq Mn} \big\{g(\eta_s^n(x))-g(\eta_s^n(x+1))\big\} \big(\nabla_x^n G\big)^2 \Delta_x^n W ds\Big)^2 \Big] \leq \\
	\leq \frac{20t}{n} \sum_{|x|\leq Mn} \big(\nabla_x^n G\big)^4 \big(\Delta_x^n W\big)^2.
\end{multline*}

Remember that $\nabla_x^n G$ is close to $H(x/n)$. This approximation is uniform in bounded intervals, in view of (\ref{ec10}). The function $H$ is bounded in bounded intervals, since it has a derivative in $\mc L^2(\bb R)$ (which is equal to the function $\mc L G$). Therefore, the last sum is bounded by $C(G) t \{W(M)-W(-M)\}^2/n$, which goes to $0$ as $n \to \infty$. Fix a number $l >0$. For notational convenience, assume that $l$ is a divisor of $n$. Required changes if it is not the case are straightforward. Repeating the computations above a finite number of times, we see that we can introduce a space average in the sum of the lemma. In other words, it is enough to show that
\[
\bb E^n\Big[\Big( \int_0^t \sum_{\substack{x:x/l \in \bb Z \\ |x|\leq Mn}} \Big( \frac{1}{l} \sum_{i=1}^l \big\{g(\eta_s^n(x+i)-\varphi(\rho)\big\} \Big) \big(\nabla_x^n G\big)^2 \Delta_x^n W ds\Big)^2 \Big]
\]
goes to $0$ as $n \to \infty$. But now the variance with respect to $\nu_\rho$ of the sum inside the integral is equal to
\[
\int \Big( \frac{1}{l} \sum_{i=1}^l \big\{g(\eta(i)-\varphi(\rho)\big\} \Big)^2 \nu_\rho(d\eta)
	\sum_{|x|\leq Mn/l}  \big(\nabla_x^{n/l} G\big)^4 \big(\Delta_{lx}^n W\big)^2,
\]
which now is of order $1/l$. Therefore, sending $l \to \infty$ after $n$ and before $M$, we complete the proof of the lemma. 
\end{proof}

The rest of the proof follows like in Section \ref{s4.2}, so we omit it. We just point out that this omitted part of the proof is simplified by the fact that we have constructed $\mc F$ in such a way that $\mc L_W$ is continuous in $\mc F$.

\subsection{The Boltzmann-Gibbs principle}

In this section we give an outline of the proof of Theorem \ref{t7}. As we mentioned before, the proof follows closely the ideas introduced by Chang \cite{Cha}, and we adopt here the proof in Chapter 11.1 of \cite{KL}. 

As we said before, the idea is to decompose the macroscopic fluctuations of $g(\eta_s^n(x))$ into two components: one given by the projection over the conserved quantities and another one orthogonal to the space of conserved quantities in a proper sense. The intuition is simple. Consider the subspace $\mc H_\rho^0$ of $\mc L^2(\nu_\rho)$ consisting of local functions $h$ satisfying $\int h d\nu_\rho =0$. An example of function in $\mc H_\rho^0$ is $\eta(x)-\rho$ for any $x \in \bb Z$. Another example is $g(\eta(x))-\varphi(\rho)$. A simple way to generate lots of local functions in $\mc H_\rho^0$ is to take $L f$ for $f$ local, where $L$ is the generator of the dynamics. What is remarkable, is that the space $\mc H_\rho^0$ can be equipped with a norm in such a way that $\mc H_\rho^0$ is the orthogonal direct sum of the subspace generated by $\{Lf; f \text{ local}\}$ and $\{\eta(x)-\rho; x \in \bb Z\}$. This decomposition is moreover continuous with respect to the energy estimate in Theorem \ref{t6}. Fortunately, due to the so-called {\em gradient condition} \cite{KL} of the zero-range process, we do not need to prove this decomposition, but only to follow the intuition given by it. 

Therefore, the idea is to find some function $f$ such that
\[
g(\eta(x)) -\varphi(\rho) -\varphi'(\rho)(\eta(x)-\rho) = \tau_x L f(\eta)
\]
plus an error term small in some sense. Here we have written $\tau_x f(\eta)= f(\tau_x \eta)$ and $\tau_x \eta(z)= \eta(x+z)$, that is, $\tau_x$ is the standard shift by $x$. We will see that the space-time fluctuations of $L f(\eta)$ are small, and that the error term will be small with respect to $\nu_\rho$, when averaged over small boxes of fixed size. 

We start with the fluctuations of $L f$. Fix an intermediate scale $l$. We will drop the dependence on $l$ from the notation, unless stated explicitly. For simplicity, assume that the function $H$ has a support contained in $(0,1)$. Define $x_i = il$ and define the generators $L^i$, $L^{i,0}$ acting on local functions $f: \Omega \to \bb R$ by
\begin{multline*}
L^i f(\eta) = \sum_{x=x_i}^{x_{i+1}-1} \frac{n}{\Delta_x^n W} 
	\Big\{ g(\eta(x)) \big[f(\eta^{x,x+1})-f(\eta)\big]  \\
		+g(\eta(x+1)) \big[f(\eta^{x+1,x})-f(\eta)\big] \Big\},
\end{multline*}
\begin{multline*}
L^{i,0} f(\eta) = \sum_{x=x_i}^{x_{i+1}-1} 
	\Big\{ g(\eta(x)) \big[f(\eta^{x,x+1})-f(\eta)\big]  \\
		+g(\eta(x+1)) \big[f(\eta^{x+1,x})-f(\eta)\big] \Big\}.
\end{multline*}

In other words, the operator $L^i$ is the generator of the dynamics restricted to the box $\Lambda_i=\{x_i,\dots,x_{i+1}\}$ and $L^{i,0}$ is the generator of a zero-range process in the same box, but with {\em uniform} transition rates. 

\begin{lemma}
\label{l9}
For any Lipschitz function $f: \bb N_0^{\Lambda_i} \to \bb R$ we have
\[
\lim_{n \to \infty} \bb E^n\Big[\Big( \int_0^t \frac{1}{n^{1/2}} \sum_{i=0}^{n/l-1} H(x_i) L^{i,0} \tau_{x_i} f(\eta_s^n)ds\Big)^2\Big] =0.
\]
\end{lemma}

\begin{proof}
By Proposition \ref{p5}, the expectation is bounded by $20t||F_n||_{-1,n}^2$, where we have used $F_n$ as a shorthand for the sum inside the integral. Now remember the variational formula for $||F_n||_{-1,n}^2$:
\begin{equation}
\label{ec11}
||F_n||_{-1,n}^2 = \sup_{h \in \mc L^2(\nu_\rho)} \big\{ 2\<h,F_n\>_\rho -\<h,-L_n h\>_\rho\big\}.
\end{equation}

We have the following relations between the different Dirichlet forms:
\[
\sum_{i=1}^{n/l} \<h, -L^ih\>_\rho \leq \<h, -L_n h\>_\rho,
\]
\[
\<h, -L^{i,0} h\>_\rho \leq \frac{W(x_{i+1}/n)-W(x_i/n)}{n} \<h, -L^ih\>_\rho.
\]

At this point, the best estimate we can have for the difference involving $W$ is $W(x_{i+1}/n)-W(x_i/n) \leq W(1)-W(0)$. But the extra factor $1/n$ will prove to be useful. Now we bound each one of the terms in $\<h,F_n\>$ separately. By Cauchy-Schwartz inequality,
\begin{align*}
 \<h, L^{i,0} \tau_{x_i} f\>_\rho
 	&\leq \frac{1}{2\beta_i} \<h, -L^{i,0} \tau_{x_i} f\>_\rho \\
	&\leq \frac{W(1) -W(0)}{2 n \beta_i} \<h,-L^ih\>_\rho +\frac{\beta_i}{2} \<\tau_{x_i} f, -L^{i,0} \tau_{x_i} f\>_\rho.
\end{align*}

In (\ref{ec11}), this term is multiplied by $H(x_i/n)/n^{1/2}$. Therefore, choosing $\beta_i = |H(x_i/n)|(W(1)-W(0))/n^{3/2}$, we can cancel the term involving $h$ in (\ref{ec11}) to get the bound
\[
 ||F_n||_{-1,n}^2 \leq \frac{W(1)-W(0)}{n^2} \sum_{i=1}^{n/l} |H(x_i/n)| \<\tau_{x_i}f, -L^{i,0} \tau_{x_i} f\>_\rho.
\]

The whole point of introducing $L^{i,0}$ is that this operator is {\em translation invariant}. As we mentioned before, the space averaging is crucial in order to obtain the Boltzmann-Gibbs principle. Due to this translation invariance, all the terms of the form $\<\tau_{x_i} f, -L^{i,0} \tau_{x_i} F\>_\rho$ are equal. We conclude that the expectation in the statement of the lemma is bounded by $C/n$ for some constant depending only on $f$ and $H$, which proves the lemma. Notice that in the homogeneous case (when the dynamics is already translation-invariant) we would have obtained a bound of order $C/n^2$ for this expectation. 
\end{proof}

The point now is that time integration has already played its role, and now only the spatial properties of the invariant measure $\nu_\rho$ are needed to continue. Define $v_x(\eta) = g(\eta(x)) -\varphi(\rho) - \varphi'(\rho)(\eta(x)-\rho)$. We have the following result:

\begin{proposition}
 \label{p6}
 \[
  \lim_{l \to \infty} \inf_{f} \limsup_{n \to \infty} \int \Big(\frac{1}{n^{1/2}} \sum_{i=1}^{n/l} \sum_{j=x_i+1}^{x_{i+1}} \big(v_j(\eta) - L^{i,0} \tau_{x_i} f\big) H(x_i/n)\big)^2 d\nu_\rho(\eta) =0,
 \]
where the infimum is over all the functions $f: \bb N_0^{\Lambda_l} \to \bb R$.
\end{proposition}

Observe that the statement of this proposition depends only on the nature of the invariant measure $\nu_\rho$ and the operator $L^{i,0}$ which does not carry any information about the environment. A detailed proof of this proposition can be found in \cite{KL}.

Now the proof of Theorem \ref{t7} is essentially finished. The scheme is the following. First introduce a spatial average, substituting $H((x_i+j)/n)$ by $H(x_i/n)$, for $j=1,\dots,l$. This introduces an error that vanishes as $n \to \infty$ and then $l \to \infty$ due to the uniform continuity of $H$ in the interval $[0,1]$. Then substract to each block the term $L^{i,0} \tau_{x_i} f$. This can be done in view of Lemma \ref{l9}. And then put the expectation inside the time integration using Cauchy-Schwartz inequality at the cost of a multiplicative constant $t^2$. At this point we have arrived exactly to Proposition \ref{p6}, which end the proof of Theorem \ref{t7}.

\subsection{Annealed density fluctuations}

In this section we discuss Corollary \ref{c1}. As we already said, we do not expect to have a functional central limit theorem for $\mc Y_\cdot^n$, since we do not have a good set of test functions not depending on $W$. For continuous functions $G: \bb R \to \bb R$ of bounded support, the random variables $\mc Y_t^W(G)$ are well defined, but even for the generalized Ornstein-Uhlenbeck process based on the usual Laplacian $\Delta$, $\mc Y_t$ is not a well defined process in the set of Radon measures, which is the dual space corresponding to continuous test functions of bounded support. Therefore, smoother test functions are needed to define $\mc Y_t$ properly. In \cite{HS}, the Sobolev space $\mc H_{3+d}$ is used as space of test functions in order to construct $\mc Y_t$. In \cite{KL}, the construction is carried out using the space $\mc H_{1+d/2}$ as test space. In any case, some smoothness of test functions is required. In our case, the very concept of smoothness changes with $W$. 

Let $\{a_n\}_n$, $\{b_n\}_n$ be two sequences of random variables defined in the same probability space. Denote by $(b_n|a_n)$ the (random) value of $b_n$ conditioned to the value of $a_n$. Now we make a simple observation.

\begin{lemma}
\label{l10}
Assume that there are two random variables $a$, $b$ such that $a_n \to a$ a.s. as $n \to \infty$ and such that $(b_n|a_n) \to (b|a)$ in distribution as $n \to \infty$. If the random variables $\{b_n\}_n$ are $\bb R^d$-valued (with $d<+\infty$), then the random vector $(a_n,b_n)$ converges in distribution to $(a,b)$.
\end{lemma}

\begin{proof}
Let $F$ be a bounded, continuous function of two variables. We want to prove that $E[F(a_n,b_n)]$ converges to $E[F(a,b)]$ as $n \to \infty$. By the dominated convergence theorem, and since $(b_n|a_n) \to (b|a)$, we see that
\[
E[F(a,b_n)] =E[E[F(a,b_n)|a_n]] \xrightarrow{n \to \infty} E[E[F(a,b)]|a] = E[F(a,b)].
\]

Now we just need to prove that $E[F(a_n,b_n)-F(a,b_n)]$ goes to $0$ as $n \to \infty$. Since $F$ is continuous, it is uniformly continuous in compact sets. Therefore, it is enough to prove that $\{(a_n,b_n)\}_n$ is tight. Since $\{a_n\}_n$ is convergent, it is automatically tight. We are left with tightness of $\{b_n\}_n$. At this point we need to use the fact that $b_n \in \bb R^d$. More precisely, we need $\{b_n\}_n$ to be defined in a $\sigma$-compact space. Of course, this is the case for $\bb R^d$, since for example $\bb R^d$ is the increasing union of the compact sets $K_l = \{x \in \bb R^d; |x|\leq l\}$, $l \in \bb N$. 

For any fixed realization of $\{a_n\}_n$, $\{(b_n|a_n)\}_n$ is tight. For any $\epsilon >0$, to each realization we associate a compact set $K$ such that $P((b_n|a_n) \notin K) <\epsilon$. Notice that any $K$ is contained in one of the sets $K_l$. Therefore, we have chosen an increasing sequence of sets $A_l$ on the underlying probability space for which the compact set chosen above is contained in $K_l$. Taking $l$ large enough, we obtain tightness for $\{b_n\}_n$. 
\end{proof}

Now it is clear on which sense Corollary \ref{c1} holds. Take a finite collection of continuous functions $G_1,\dots,G_l$ with bounded support and a finite collection of times $t_1 \leq \dots \leq t_l$. By Theorem \ref{t3} and an approximation procedure, the vector $(\mc Y_{t_1}^{n,W}(G_1),\dots,\mc Y_{t_l}^{n,W}(G_l))$ converges in distribution to $(\mc Y_{t_1}^{W}(G_1),\dots,\mc Y_{t_l}^{W}(G_l))$. And now by Lemma \ref{l10}, we conclude that the vector $(\mc Y_{t_1}^{n}(G_1),\dots,\mc Y_{t_l}^{n}(G_l))$ converges in distribution to the vector $(\mc Y_{t_1}^{\text{ann}}(G_1),\dots,\mc Y_{t_l}^{\text{ann}}(G_l))$.

\section{Current fluctuations}
\label{s6}
In this section we prove Theorem \ref{t4}. Therefore, from now on we take $d=1$. We will treat the superdiffusive and subdiffusive cases separately. We start explaining a generalization of the original idea of Rost and Vares \cite{RV}, which works in both cases.

Let us define $G_l(x) = (1-x/l)^+ \mathbf{1}(x \geq 0)$, where $y^+$ denotes the positive part of $y$. The sequence $\{G_l\}_l$ converges uniformly in compacts to the Heaviside function $H_0$ defined in Section \ref{s4.5}. This convergence also holds in another ``energy'' sense. Of course $G_l$ is not an admissible test function (neither in the superdiffusive case nor in the subdiffusive case). In the other hand, $\mc Y_t(G_l)$ is well defined for any $t \geq 0$, since $\mc Y_t(\cdot)$ is a white noise in $\mc F$ and therefore it can be continuously extended to $\mc L^2(\bb R)$. By a similar reasoning, the joint distributions  $\{\mc Y_{t_1}(G_l),\dots,\mc Y_{t_l}(G_l)\}$ are well defined. What is not well defined is $\mc Y_\cdot(G_l)$ as a real-valued, right-continuous process; this is one of the reasons why we only obtain convergence of finite distributions in Theorem \ref{t4}.
Let us recall the decomposition
\[
\mc Y_t(G) -\mc Y_0(G) = \varphi'(\rho) \int_0^t \mc Y_s(\mc L G)ds + \mc M_t(G).
\]

Using the energy estimate in Theorem \ref{t6} and also the explicit expression for the quadratic variation of $\mc M_t(G)$, we see that the left-hand side of this identity is continuous with respect to the seminorm $\mc E(G,G)^{1/2}$. The idea is that $\{G_l\}_l$ converges to $H_0$ under this seminorm, which allows to define the random variable $\mc Y_t(H_0)-\mc Y_0(H_0)$ by continuity. It is important to operate carefully with this random variable, since the terms $\mc Y_t(H_0)$ and $\mc Y_0(H_0)$ are {\em not} well defined. In order to obtain convergence of the rescaled current to this random variable, we will see that the sequence $\{\mc Y_t^n(G_l)-\mc Y_0^n(G_l)\}_n$ converges to the rescaled current, {\em uniformly} in $n$. In particular the limits $n \to \infty$ and $l \to \infty$ will be exchangeable.

\subsection{The superdiffusive case}

In Theorem \ref{t6}, we proved the energy estimate for the subdiffusive case. The proof for the superdiffusive case is exactly the same. For the reader's convenience, we restate here the energy estimate.

\begin{theorem}
\label{t8}
For any function $H: \bb R \to \bb R$,
\[
\bb E^n \Big[\Big( \int_0^t \frac{1}{n^{1/2}} \sum_{x \in \bb Z} \big(g(\eta_s^n(x)) -\varphi(\rho)\big) \mc L_n H(x/n) ds \Big)^2 \Big] \leq 5t \varphi(\rho)^{-1} \mc E_n(H,H).
\]
\end{theorem}

This estimate, together with Theorem \ref{t2} allows to obtain part $i)$ of Theorem \ref{t4}. For $x \in \bb Z$ define $J_x^n(t)=n^{-1} J_x(tn^{2\alpha})$. First notice that
\begin{align*}
J_0^n(t) - \mc Y_t^n(G_l) + \mc Y_0^n(G_l)
	&= \int_0^t \frac{1}{n^{1/2}} \sum_{x \in \bb Z} \big(g(\eta_s^n(x)) -\varphi(\rho)\big) \mc L_n(H_0-G_l) ds \\
	&+ \mc M_t^n(H_0-G_l).
\end{align*}

In principle, $\mc M_t^n(H_0-G_l)$ is not well defined, since $H_0-G_l$ is not in $\mc L^2(\bb R)$. However, using the relations $\eta_s^n(x)-\eta_0^n(x) = J_x^n(t)$ and the explicit formula for the compensator of $J_x^n(t)$, this formula is easily justified. Using Theorem \ref{t8} and the explicit form of the quadratic variation $\<\mc M_t^n(H_0-G_l)\>$ we obtain that
\[
\bb E^n\Big[ \Big( J_0^n(t)- \mc Y_t^n(G_l) + \mc Y_0^n(G_l)\Big)^2\Big]
	\leq c(\rho) t \mc E_n(H_0-G_l,H_0-G_l)
\]
for some constant $c(\rho)$ not depending on $l$, $n$ nor $t$. A simple computation shows that $\mc E_n(H_0-G_l,H_0-G_l) \leq c l^{1-\alpha}$ for some constant $c$, depending only on $\alpha$. We conclude that $\mc Y_t^n(G_l)-\mc Y_0^n(G_l)$ converges to $J_0^n(t)$ as $l \to \infty$ {\em uniformly} in $n$. 

At this point, we need to justify the limit $\mc Y_t^n(G_l) \to \mc Y_t(G_l)$ as $n \to \infty$, which does not follows from the convergence of $\mc Y_\cdot^n$, since $G_l$ is not a test function. But this is elementary, since the sequence $\{\mc Y_t^n(\cdot)\}_n$ is uniformly continuous under convergence in $\mc L^2$. 

Repeating these computations for the process $\mc Y_\cdot$, we see that
\[
E\big[ \big(\mc Y_t(G_l-G_{l+m})-\mc Y_0(G_l-G_{l+m})\big)^2\big] \leq c t l^{1-\alpha}
\]
for a constant $c$ depending only on $\alpha$ and $\rho$. In particular, the sequence $\{\mc Y_t(G_l)-\mc Y_0(G_l)\}_l$ is a Cauchy sequence and $\mc Y_t(H_0)-\mc Y_0(H_0)$ is well defined as the limit of this sequence. Now we can exchange the limits in $n$ and $l$ to get the following result.

\begin{theorem}
\label{t9}
For the zero-range process with long jumps and $\alpha \in (1,2)$, starting from the initial distribution $\nu_\rho$,
\[
\lim_{n \to \infty} \frac{1}{n} J_0(tn^{2\alpha}) = \mc Y_t(H_0) -\mc Y_0(H_0),
\]
in the sense of finite-dimensional distributions, where the process $\mc Y_t(H_0)-\mc Y_0(H_0)$ is defined as above. 
\end{theorem}

Although we have proved convergence just for one-dimensional distributions, the same arguments give convergence of $k$-dimensional distributions by considering properly defined $k$-dimensional vectors. The only point left in order to close the proof of part $i)$ in Theorem \ref{t4} is to identify the process $\mc Y_t(H_0)-\mc Y_0(H_0)$ as a fractional Brownian motion. Since the finite-dimensional distributions of $\mc Y_\cdot$ are all mean-zero and Gaussian, $\mc Y_t(H_0)-\mc Y_0(H_0)$ inherits this property and it is a mean-zero, Gaussian process. Given a Gaussian, mean-zero process with stationary increments $\mc X(t)$, its variance $E[\mc X(t)^2]$ identifies the process. But a simple scaling argument shows that the distributions of $J_0^n(t)$ and $t^{1/2\alpha}J_0^n(1)$ are the same. Therefore, $E[(\mc Y_t(H_0) -\mc Y_0(H_0))^2]=ct^{1/\alpha}$, and $\mc Y_t(H_0)-\mc Y_0(H_0)$ is a fractional Brownian motion of Hurst exponent $H=1/2\alpha$.

\subsection{The subdiffusive case}

A simple computation shows that the choice of $\{G_l\}_l$ made in the previous section does not work in the subdiffusive case. The good choice is $G_l(x) = (1-W(x)/W(l))^+ \mathbf{1} (x\geq 0)$. In fact, after some computations we see that
\[
 \mc E_n(H_0-G_l, H_0-G_l) = \frac{1}{W(l)},
\]
\[
 \mc E_n(G_m-G_l,G_m-G_l) = \frac{|W(m)-W(l)|}{W(m)W(l)}.
\]

In particular, those quantities go to 0 as $l \to 0$, uniformly in $n$, since $W(l) \to 0$ as $l \to \infty$. Repeating the arguments of the previous section, it is easy to obtain the following result, which is just a restatement of Theorem \ref{t4}, part $ii)$.

\begin{theorem}
 \label{t10}
 For the zero-range process with environment $W$ and initial distribution $\nu_\rho$,
 \[
  \lim_{n \to \infty} \frac{1}{n} J_0(t n^{2+2/\alpha}) = \mc Y_t^W(H_0) -\mc Y_0^W(H_0),
 \]
in the sense of convergence of finite-dimensional distributions.
\end{theorem}

In this theorem, $\mc Y_\cdot^W$ is the solution of (\ref{ec4.5}) and $\mc Y_t^W(H_0)-\mc Y_0^W(H_0)$ is defined by continuity as the limit of $\mc Y_t^W(G_l)-\mc Y_0^W(G_l)$ when $l \to \infty$. The idea now is to use Lemma \ref{l10} to get the annealed result stated in part $iii)$ of Theorem \ref{t4}. Despite the fact that Corollary \ref{c1} is weaker than Theorem \ref{t2}, in order to use Lemma \ref{l10} only convergence in the sense of finite-dimensional distributions is needed. We state the corresponding result for $J_0(t)$ as a corollary.

\begin{corollary}
 \label{c2}
 Under the annealed law $P \otimes \bb P^n$,
 \[
  \lim_{n \to \infty} \frac{1}{n} J_0(t n^{2+2/\alpha}) = \mc X(t),
 \]
where $\mc X(\cdot)$ has the same distribution of $\mc Y_\cdot^W(H_0) -\mc Y_0(H_0)$ with respect to the annealed law.
\end{corollary}

Remember that any result in distribution with respect to $P \otimes \bb P^n$ has a counterpart for the laws $P^\xi \otimes \bb P^n$. In particular, in order to prove part $iii)$ of Theorem \ref{t4}, we are just left to prove that the process $\mc X(t)$ defined above is a fractional Brownian motion of Hurst index $\alpha/(2+2\alpha)$. The self-similarity follows at once from the scale invariance of $\{\xi_x\}_x$ and $J_0^n(t)$. Of course, for each fixed $W$, $\mc X^W(t)$ is a Gaussian random variable. But it is not clear why the averaged law of $\mc X^W(t)$ is also Gaussian. The key point is that the ``Gaussian character'' of $\mc X^W(t)$ does not depend on $W$. In fact, by the  construction in Section \ref{s2}, the process $\mc Y_\cdot^W$ has Gaussian finite distributions for any $W$. In a more formal setting, we can construct the processes $(W, \mc Y_\cdot^W)$ in a probability space big enough such that there are random variables $\sigma(W,t)$, $\zeta(W,t)$ such that the family $\{\zeta(W,t)\}_W$ is i.i.d. with common distribution equal to a normal distribution of mean zero and variance 1, and such that $\mc X^W(t) = \sigma(W,t)\zeta(W,t)$ for any $W$, $t$. In other words, for a fixed time $t$, all the dependence of $\mc X^W(t)$ in $W$ is encoded on its variance. From this decomposition plus the fact that the sum of two independent, Gaussian variables is also Gaussian, it is clear that $\mc X(t)$ is Gaussian. Since any self-similar, Gaussian process is a fractional Brownian motion, we have finished the proof of Theorem \ref{t4}, part $iii)$.

\section{The simple exclusion process with variable diffusion coefficient}
\label{s7}

In this section we obtain a central limit theorem for a tagged particle in the simple exclusion process with variable diffusion coefficient as an application of the results on this article in the subdiffusive case.

Consider a system $\{x_i(t); i \in \bb Z\}$ of continuous-time, interacting random walks on the one-dimensional lattice $\bb Z$, and let $\lambda =\{\lambda_i; i \in \bb Z\}$ be a given sequence of positive numbers. The dynamics of these particles is the following. The particle $x_i$ waits an exponential time of rate $2 \lambda_i$, at the end of which it chooses one of its two neighbors with equal probability. If there is no other particle $x_j$ at the chosen site at that moment, the particle jumps to that site; if the site is occupied the particle stays where it is. In any case a new exponential time of rate $2\lambda_i$ starts afresh. This happens independently for each particle. Notice that particles only interact through the so-called {\em exclusion rule}. This dynamics corresponds to a Markov process $\mathbf x(t)$ defined on the state space $\Omega_{ex}= \{\mathbf x \in \bb Z^{\bb Z}; x_i \neq x_j \text{ if } i \neq j\}$ and generated by the operator
\begin{align*} 
L^{ex} f(\mathbf x)
	& \sum_{i \in \bb Z} \lambda_i \Big\{ \mathbf 1(\mathbf x+e_i \in \Omega_{ex}) \big[f(\mathbf x +e_i) - f(\mathbf x)\big] \\
		&+ \mathbf 1 (\mathbf x -e_i \in \Omega_{ex} ) \big[f(\mathbf x -e_i) -f(\mathbf x) \big] \Big\}, 
\end{align*}
where $\{e_i;i \in \bb Z\}$ is the canonical basis in $\bb Z^{\bb Z}$. The number $\lambda_i$ is interpreted as the {\em diffusion coefficient} of particle $x_i$. We call the process $\mathbf x(t)$ the simple exclusion process with variable diffusion coefficient. Notice that when the sequence $\{\lambda_i\}_i$ is constant, $\mathbf x(t)$ is simply the usual simple exclusion process, but labeled: the so-called {\em stirring process}. Notice that the relative ordering of particles is preserved by the dynamics. Then, without loss of generality, we assume that $x_i(t) <x_{i+1}(t)$ for any $i \in \bb Z$ and any $t \geq 0$.  We call particle $x_0$ the {\em tagged particle}. We want to study the asymptotic behavior of $x_0(t)$. In particular, we want to obtain a central limit theorem for this particle. Again without loss of generality, we can assume that $x_0(0)=0$, that is, that the tagged particle is at the origin for $t=0$. 

The relation between this model and the model studied in this article comes from the following observation. Define, for each $i \in \bb Z$ and each $t \geq 0$, $\eta_t(i) = x_{i+1}(t)-x_i(t)-1$. A simple computation shows that in fact the process $\eta_t = \{\eta_t(i); i \in \bb Z\}$ is a zero-range process with interaction rate $g(n) = \mathbf 1(n>0)$ and transition probability $p(i,i-1) =p(i-1,i) = \lambda_i$. Remarkably, the position $x_0(t)$ of the tagged particle is equal to the current $J_0(t)$ through the bond $\<-1,0\>$ for $\eta_t$. Now take $\{\lambda_i\}_i$ as a realization of a sequence of i.i.d. random variables with common distribution $\lambda$, such that $\lambda^{-1}$ is in the domain of attraction of an $\alpha$-stable law, $0<\alpha<1$. Then the process $\eta_t$ is just the zero-range process with random environment defined in Section \ref{s3.4}. 

Consider the {\em environment process} $\mathbf X^0(t)$ defined by $x^0_i(t)= x_i(t) -x_0(t)$. In other words, $\mathbf X^0(t)$ corresponds to the relative position of the particles with respect to the tagged particle. Knowing the relation between $\mathbf X^0(t)$ and $\eta_t$, it is not difficult to identify the invariant measures for $\mathbf X^0(t)$. In fact, since the invariant measures of $\eta_t$ are products of geometric distributions (for our particular choice of $g(\cdot)$), we conclude that the distance between particles is a geometric distribution. Let us denote by $\mu_\rho$ the invariant measure associated to $\mathbf x(t)$ corresponding to $\nu_\rho$, that it, $\mu_\rho\{x_{i+1}-x_i=k\}= \nu_\rho\{\eta(i) = k-1\}$. It is not difficult to see that the density of particles according to $\mu_\rho$ is equal to $1/(1+\rho)$. 

Just translating the results of Theorem \ref{t4} into the language of the simple exclusion process with variable diffusion coefficient, we obtain the following result.

\begin{theorem}
 \label{t11}
 Let $\mathbf x(t)$ be a simple exclusion process with i.i.d., $\alpha$-stable diffusion coefficients. Assume that $\mathbf X^0(t)$ is distributed according to the invariant measure $\mu_\rho$. Then, the position $x_0(t)$ of the tagged particle satisfies the following central limit theorem:
 \[
  \lim_{n \to \infty} \frac{1}{n} x_0(t n^{2+2/\alpha}) = \mc X(t)
 \]
in the sense of convergence of finite-dimensional distributions, when averaged over the diffusion coefficients. Here $\mc X(t)$ is a fractional Brownian motion of Hurst exponent $\alpha/(2+2\alpha)$.
\end{theorem}

\section*{Acknowledgements}
Part of this work was done during the author's 
stay at the Institut Henri Poincar\'e, Centre 
Emile Borel. The author thanks this institution 
for hospitality and support.
 M. J. was supported by the Belgian Interuniversity Attraction Poles Program P6/02, through the network NOSY (Nonlinear systems, stochastic processes and statistical mechanics). M.J. would like to thank T. Ambj\"ornsson for sending an early version of \cite{ASLL} and for helpful correspondence.

\end{document}